\documentclass[12pt]{article}

\usepackage{xypic}
\usepackage{amsmath,amstext,amsbsy,amssymb,amsthm}
\usepackage{mine}
\usepackage{graphicx}
\usepackage{color, colortbl}
\usepackage{comment}
\definecolor{Gray}{gray}{0.9}
\begin{document}
\selectfont
\define\div{\big|}
\define\ndiv{\not\div}
\define\HG{\operatorname{Hol}(G)}
\define\NHG{\operatorname{NHol}(G)}
\define\QHG{\operatorname{QHol}(G)}
\define\XHG{\operatorname{XHol}(G)}
\define\QHN{\operatorname{QHol}(N)}
\define\QHC{\operatorname{QHol}(C_{p^n})}
\define\lG{\lambda(G)}
\define\lg{\lambda(g)}
\define\rG{\rho(G)}
\define\lx{\lambda(x)}
\define\lxi{\lambda(x)^{-1}}
\define\lt{\lambda(t)}
\define\lti{\lambda(t)^{-1}}
\define\cg{\cellcolor{Gray}}
\define\conj[#1][#2]{#2#1#2^{-1}}
\define\xconj[#1][#2]{#2^{-1}#1#2}
\define\zetabar{\overline{\zeta}}
\define\half{\mathfrak{h}}
\define\ds{\displaystyle}
\def\keywords#1{\def\@keywords{#1}}
\parskip=0.125in
\title{Mutually Normalizing Regular Permutation Groups and Zappa-Sz\'ep Extensions of the Holomorph}
\date{\today}
\author{Timothy Kohl\\
Department of Mathematics and Statistics\\
Boston University\\
Boston, MA 02215\\
tkohl@math.bu.edu}
\maketitle
\begin{abstract}
For a group $G$, embedded in its group of permutations $B=Perm(G)$ via the left regular representation $\lambda:G\rightarrow B$, the normalizer of $\lambda(G)$ in $B$ is $\operatorname{Hol}(G)$, the holomorph of $G$. The set $\mathcal{H}(G)$ of those regular $N\leq \operatorname{Hol}(G)$ such that $N\cong G$ and $\operatorname{Norm}_B(N)=\operatorname{Hol}(G)$ is keyed to the structure of the so-called multiple holomorph of $G$, $N\!Hol(G)=\operatorname{Norm}_B(\operatorname{Hol}(G))$, in that $\mathcal{H}(G)$ is the set of conjugates of $\lambda(G)$ by $N\!Hol(G)$. We wish to generalize this by considering a certain set $\mathcal{Q}(G)$ consisting of regular subgroups $M\leq \operatorname{Hol}(G)$, where $M\cong G$, that contains $\mathcal{H}(G)$ with the property that its members mutually normalize each other. This set will generally give rise to a group $Q\!\operatorname{Hol}(G)$ which we will call the quasi-holomorph of $G$, where the orbit of $\lambda(G)$ under $Q\!\operatorname{Hol}(G)$ is $\mathcal{Q}(G)$. The multiple holomorph is a group extension of $\operatorname{Hol}(G)$ and the quasi-holomorph will contain $N\!Hol(G)$, but, when larger than $N\!Hol(G)$, is frequently a Zappa-Sz\'ep product with the holomorph.
\end{abstract}
\par
\renewcommand{\thefootnote}{}
\section{Introduction}
Throughout this discussion, all groups will be assumed to be finite. For a group $G$, the holomorph $\operatorname{Hol}(G)$ is a classical object that can be defined in two different ways. The first is as the semi-direct product
$$
G\rtimes \operatorname{Aut}(G)
$$
where $(g,\alpha)(h,\beta)=(g\alpha(h),\alpha\beta)$ with $G$ embedded as the normal subgroup $\hat{G}=\{(g,I)\ |\ g\in G\}$. The other formulation involves the left regular representation of $G$. If $B=Perm(G)$ then $\lambda:G\rightarrow B$ given by $\lambda(g)(h)=gh$ is an embedding and one defines the holomorph of $G$ to be the normalizer $\operatorname{Norm}_B(\lambda(G))$. This normalizer is isomorphic to the aforementioned semi-direct product, by virtue of the following result \cite[Theorem 6.3.2]{Hall1959} which can be found in Hall's book which we paraphrase here.
\begin{theorem}
For $G$ embedded in $B=Perm(G)$ as $\lambda(G)$
$$
\operatorname{Hol}(G)=\operatorname{Norm}_B(\lambda(G))=\rho(G)A(G)
$$
where $\rho:G\rightarrow B$ is given by $\rho(g)(h)=hg^{-1}$ (the right regular representation) and $A(G)=\{\pi\in \operatorname{Hol}(G)\ |\ \pi(id_G)=id_G\}=\operatorname{Aut}(G)$ namely those elements of the normalizer which fix the identity element of $G$.
\end{theorem}
There are a number of important observations to make. The first is that both $\lambda(G)$ and $\rho(G)$ are embedded in $B$ as {\it regular} subgroups, namely they act transitively and the stabilizer of any $g\in G$ is the identity. Also, $\rho(G)$ is clearly normalized by $\operatorname{Aut}(G)$. Since the stabilizer of any element is trivial, in the above formulation, $\operatorname{Hol}(G)$ is very naturally a split extension of $\rho(G)$ by $A(G)$ since $\rho(G)\cap A(G)=\{id_B\}$. One should also note that $\lambda(G)$ and $\rho(G)$ clearly centralize each other. 
Moreover, if we have the inner automorphism $c_g$ for $g\in G$ then $c_g[x]=\lambda(g)\rho(g)[x]$. As such, if $\alpha\in \operatorname{Aut}(G)$ then $$\rho(g)\circ\alpha[x]=\lambda(g^{-1})\circ(c_g\circ\alpha)[x]$$ for each $x\in G$. The end result is that $\operatorname{Hol}(G)=\rho(G)\operatorname{Aut}(G)=\lambda(G)\operatorname{Aut}(G)$ which means that $\operatorname{Hol}(G)$ could equally well be defined to be the normalizer of $\rho(G)$. Going further, since $\lambda$ and $\rho$ are injective then clearly $\lambda(G)\cong\rho(G)$, although they are not necessarily equal, unless $G$ is abelian. Indeed, it is not difficult to verify that $\lambda(G)\cap\rho(G)\cong Z(G)$.\par
Beyond $\lambda(G)$ and $\rho(G)$ one can consider what other regular subgroups of $B$, isomorphic to $G$, have the same normalizer. 
\begin{definition}
$$
\mathcal{H}(G)=\{N\leq B\ |N\text{ is regular}, N\cong G\text{, and }\operatorname{Norm}_B(N)=\operatorname{Hol}(G)\}$$
\end{definition}
We note that if $N\in\mathcal{H}(G)$ then one must have $N\triangleleft\operatorname{Hol}(G)$, so, in particular $N$ is a subgroup of $\operatorname{Hol}(G)$.\par
\noindent The key fact we shall use, here, and throughout, is the following fact about isomorphic regular subgroups.
\begin{proposition}\cite[p. 427]{Dixon1971} 
\label{magic}
If $N,N'$ are regular subgroups of $S_n$ which are isomorphic (as abstract groups) then they are, in fact, conjugate as subgroups of $S_n$.
\end{proposition}
If now $N\leq B$ is regular and isomorphic to $G$ where $\operatorname{Norm}_B(N)=\operatorname{Hol}(G)$, then there must exist $\beta\in B$ such that $\beta\lambda(G)\beta^{-1}=N$. What this implies is that 
$$\beta \operatorname{Norm}_B(\lambda(G))\beta^{-1}=\operatorname{Norm}_B(\beta\lambda(G)\beta^{-1})=\operatorname{Norm}_B(N)$$
since the conjugate of the normalizer is the normalizer of the conjugate.  But since $\operatorname{Norm}_B(N)=\operatorname{Norm}_B(\lambda(G))$ by assumption, then $\beta$ actually normalizes $\operatorname{Hol}(G)$ itself as a subgroup of $B$. This leads to the following.
\begin{definition}
For $G$ a group, the {\it multiple holomorph} of $G$ is 
$$\NHG=\operatorname{Norm}_B(\operatorname{Hol}(G))$$
the normalizer of the holomorph.
\end{definition}
The multiple holomorph allows one to determine those regular subgroups of $\operatorname{Hol}(G)$, isomorphic to $G$, whose normalizer is $\operatorname{Hol}(G)$. In particular if $\beta$ normalizes $\operatorname{Hol}(G)$ then the fact that $\beta Norm_B(\lambda(G))\beta^{-1}=Norm_B(\beta\lambda(G)\beta^{-1})$ yields the following, which Miller established in \cite{Miller1908}.
\begin{theorem}
\label{NHG}
For $G$ a group, $T(G)=\NHG/\operatorname{Hol}(G)$ acts as a regular permutation group on $\mathcal{H}(G)$.
\end{theorem}
\noindent What this implies is that 
\begin{align*}
\mathcal{H}(G)&=\{N\leq B\ |N\text{ is regular}, N\cong G\text{, and }\operatorname{Norm}_B(N)=\operatorname{Hol}(G)\}\\
              &=\{N\triangleleft \operatorname{Hol}(G)\ |\ N\text{ regular }, N\cong G\}\\
              &=\{\beta_i\lambda(G)\beta_i^{-1}\ |\ \beta_i\operatorname{Hol}(G)\in T(G)\} \\
\end{align*}
so that $\mathcal{H}(G)$ is the orbit $\lG$ under conjugation, with $T(G)$ acting regularly. There are instances when $\mathcal{H}(G)$ is exactly $\{\lambda(G),\rho(G)\}$ such as when $G$ is a non-abelian simple group, as shown in \cite[Theorem 4]{CarnahanChilds1999}. Also, for abelian groups, Miller \cite{Miller1908} showed that $\mathcal{H}(G)=\{\lambda(G)\}$ if $8\ndiv |G|$ which implies that $\NHG=\operatorname{Hol}(G)$. This was also established by Dalla Volta and Caranti in \cite{Caranti2017} using modern methods. Beyond these, $T(G)$ and $\mathcal{H}(G)$ have been computed for various classes of groups, such as finitely generated abelian groups in \cite{Mills1951}, dihedral and quaternionic in \cite{Kohl-Multiple-2014}, finite class 2 $p$-groups in \cite{Caranti-P-Group}, and almost simple groups in \cite{Tsang-Multiple-NYJM}, and others.\par
However, we are interested in a class of groups that contains $\mathcal{H}(G)$ but is frequently larger, and to represent this class as the orbit of $\lG$ under the action of a certain group, which contains $\operatorname{Hol}(G)$ as the stabilizer subgroup. The condition we will be exploring that defines this class is the property of being mutually normalizing.
\section{Normalizing vs Being Normalized}
For a given $G$ we shall start by considering two fundamental classes of regular permutation groups.
\begin{definition}
\begin{align*}
\mathcal{S}(G)&=\{N\leq B\ |\ N\text{ regular }, N\cong G\text{ and }N\leq \operatorname{Hol}(G)\}\\
\mathcal{R}(G)&=\{N\leq B\ |\ N\text{ regular }, N\cong G\text{ and }\lambda(G)\leq \operatorname{Norm}_B(N)\}\\
\end{align*}
\end{definition}
Of particular interest is the intersection $\mathcal{S}(G)\cap\mathcal{R}(G)$ in that it consists of those regular subgroups of $\operatorname{Hol}(G)$ which normalize and are normalized by $\lG$. One immediate consequence of this definition is the following basic observation.
\begin{lemma}
\label{SRcontainsH}
For $\mathcal{S}(G)$, $\mathcal{R}(G)$ as defined above, one has that $\mathcal{H}(G)\subseteq\mathcal{S}(G)\cap\mathcal{R}(G)$
\end{lemma}
\begin{proof}
We observe that if $N\leq \operatorname{Hol}(G)$ has normalizer equal to $\operatorname{Hol}(G)$ then $N\triangleleft \operatorname{Hol}(G)$ so that $\lG\leq \operatorname{Hol}(G)$ normalizes $N$ obviously, but $N\leq \operatorname{Hol}(G)$ means $N$ normalizes $\lG$.
\end{proof}
As reviewed in the previous section, the group $\NHG$ acts transitively on $\mathcal{H}(G)$ where the stabilizer subgroup of $\lG$ is $\operatorname{Hol}(G)$. Moreover, the quotient group $T(G)$ acts regularly on $\mathcal{H}(G)$, and in particular, each coset representative corresponds, of course, to a member of $\mathcal{H}(G)$. We note here, and will explore further later on, $\NHG$ is not necessarily a split extension of $\operatorname{Hol}(G)$ and so there may not exist any set of coset representatives which forms a subgroup of $\NHG$. However, there are many examples where this {\it is} indeed the case. Nonetheless, any set of coset representatives obviously correspond to the conjugates of $\lG$ that comprise $\mathcal{H}(G)$.\par
What we (initially) seek is a subgroup of $B$ analogous to $\NHG$, which acts transitively on $\mathcal{S}(G)\cap\mathcal{R}(G)$, where $\operatorname{Hol}(G)$ is exactly the stabilizer subgroup for $\lG$. However, we will frequently have to restrict our attention to some subset $\mathcal{X}(G)\subseteq\mathcal{S}(G)\cap\mathcal{R}(G)$. This is due to the fact that $\mathcal{S}(G)\cap\mathcal{R}(G)$ need not, as it turns out, be the orbit of $\lG$ under the action of a group containing $\operatorname{Hol}(G)$.\par
To prepare for this, we need to explore the symmetry between $\mathcal{S}(G)$ and $\mathcal{R}(G)$. The one interesting thing to observe is that the $N\in\mathcal{S}(G)$ are all subgroups of $\operatorname{Hol}(G)$ and are therefore readily enumerated if one has enough information about the structure of $\operatorname{Hol}(G)$. In contrast, the set $\mathcal{R}(G)$ is a bit more mysterious in that it consists of groups which are not necessarily subgroups of $\operatorname{Hol}(G)$. However, the 'isomorphic implies conjugate' property of regular subgroups in \ref{magic} will allow us to, at least in principle, enumerate $\mathcal{R}(G)$ given $\mathcal{S}(G)$.\par
\section{Classes of Regular Subgroups}
Although we are focusing on the sets $\mathcal{S}(G)$ and $\mathcal{R}(G)$ and their intersection, we shall prove a somewhat more general statement about classes of regular subgroups of a given degree. In particular, we will consider regular subgroups $N\leq Perm(X)$ for an arbitrary finite set $X$. Later on we shall return to the case where $X=G$.\par
For $X$ a finite set, we can take an arbitrary regular subgroup $G\leq Perm(X)$ and form the normalizer $\operatorname{Norm}_B(G)$ which by a slight abuse of notation, we shall denote $\operatorname{Hol}(G)$, which is canonically isomorphic to the classic holomorph of $G$.\par
\noindent For any pair of such regular subgroups $G_j,G_i\leq B$ we can define
\begin{align*}
S(G_j,[G_i])&=\{M\leq \operatorname{Hol}(G_j)\ |\ M\text{ is regular and $M\cong G_i$} \}\\
R(G_i,[G_j])&=\{N\leq B\ |\ N\text{ is regular and $G_i\leq \operatorname{Hol}(N)$ and $N\cong G_j$} \}\\
\end{align*}
which are two {\it complementary} sets of regular subgroups of $B$, one of which is contained within a fixed subgroup of $B$, while the other consists of subgroups of $B$ which may be widely dispersed within $B$. The class $R(G_i,[G_j])$ is of interest as, by the main theorem in \cite{GreitherPareigis1987}, it corresponds bijectively to the Hopf-Galois structures on a Galois extension $L/K$ given by a $K$-Hopf algebra $H$ where $Gal(L/K)\cong G_i$ and $H=(L[N])^{Gal(L/K)}$ where $N\cong G_j$. Another point of interest, is that, by \cite[Prop.\ A.3]{SVskewbraces}, every $N\in S(G_j,[G_i])$ corresponds to a skew brace with additive group isomorphic to $G_j$ and circle group isomorphic to $G_i$. However, unlike the case of Hopf-Galois structures and groups in $R(G_i,[G_j])$, the correspondence between the set of braces and the groups in $S(G_j,[G_i])$ is not one-to-one.\par
The enumeration of $R(G_i,[G_j])$ has been the subject of many papers in Hopf-Galois theory. In particular the fundamental relationship between $S(G_j,[G_i])$ and $R(G_i,[G_j])$ has been explored in \cite[Prop. 1]{Childs1989} by Childs, in \cite[Prop. 1]{Byott1996} by Byott, and by the author in \cite[p.540]{Kohl1998}. We present the following re-capitulation of all these ideas by showing that both $S(G_j,[G_i])$ (resp. $R(G_i,[G_j])$) give rise to a set of cosets with respect to $\operatorname{Hol}(G_i)$ (resp. $\operatorname{Hol}(G_j)$) where the unions of each collection of cosets are in bijective correspondence. This bijection we shall refer to as the reflection principle.
\begin{proposition}
\label{reflection}
If $B=Perm(X)$ and for $G_i$ and $G_j$ arbitrary regular subgroups of $B$ then
$$
|S(G_j,[G_i])|\cdot|\operatorname{Hol}(G_i)|=|R(G_i,[G_j])|\cdot|\operatorname{Hol}(G_j)|.
$$
\end{proposition}
\begin{proof}
 If $M\in S(G_j,[G_i])$ then $M\leq \operatorname{Hol}(G_j)$ and $M\cong G_i$ which implies that there exists $\beta\in B$ such that $M=\beta G_i\beta^{-1}$. We observe further that $\beta G_i\beta^{-1}\leq \operatorname{Hol}(G_j)$ if and only if $G_i\leq \operatorname{Hol}(\beta^{-1}G_j\beta)$ and therefore that $\beta^{-1}G_j\beta\in R(G_i,[G_j])$. Moreover, if we replace $\beta$ by $\beta h$ for any $h\in \operatorname{Hol}(G_i)$ then $\beta h G_i(\beta h)^{-1}=\beta G_i\beta^{-1}=M$ but that $(\beta h)^{-1}G_j\beta h$ are all (not necessarily distinct) elements of $R(G_i,[G_j])$. In parallel, any $N\in R(G_i,[G_j])$ is equal to $\alpha G_j\alpha^{-1}$ for some $\alpha$ and that replacing $\alpha$ by $\alpha k$ for any $k\in \operatorname{Hol}(G_j)$ yields the same $N$. Moreover $\alpha^{-1}G_i\alpha$ lies in $S(G_j,[G_i])$ and likewise $(\alpha k)^{-1}G_i(\alpha k)$. Note that $\beta_1G_i\beta_1^{-1}=\beta_2G_i\beta_2^{-1}$ if and only if $\beta_1\operatorname{Hol}(G_i)=\beta_2 \operatorname{Hol}(G_i)$. As such we can parameterize the elements of $S(G_j,[G_i])$ by a set of distinct cosets
$$\beta_1\operatorname{Hol}(G_i),\dots,\beta_s \operatorname{Hol}(G_i)$$
and $R(G_i,[G_j])$ by distinct cosets
$$\alpha_1 \operatorname{Hol}(G_j),\dots,\alpha_r \operatorname{Hol}(G_j)$$
The bijection we seek is as follows:
$$
\dsize\Phi:\bigcup_{k=1}^{s} \beta_k\operatorname{Hol}(G_i) \rightarrow \bigcup_{l=1}^{r} \alpha_l \operatorname{Hol}(G_j)
$$
defined by $\Phi(\beta_k h) = (\beta_k h)^{-1}$. That $\Phi(\beta_k h)$ lies in the union on the right hand side is due to the analysis given above, and it is clear that $\Phi$ is bijective and therefore the result follows.
\end{proof}
\noindent Since $|\operatorname{Hol}(G)|=|G|\cdot|\operatorname{Aut}(G)|$ we have the following, which agrees with \cite[Cor. to Prop. 1]{Byott1996}.
\begin{corollary}
$$
|S(G_j,[G_i])|\cdot|\operatorname{Aut}(G_i)|=|R(G_i,[G_j])|\cdot|\operatorname{Aut}(G_j)|.
$$
and if $G_i=G=G_j$ is a fixed regular subgroup then
$$
|S(G,[G])|=|R(G,[G])|
$$
that is, the set of regular subgroups of $\operatorname{Hol}(G)$ that are isomorphic to $G$ has the same cardinality as the set of those regular subgroups of $B$ which are isomorphic to $G$ and normalized by $G$.
\end{corollary}
We note briefly that for $X=G$, our $\mathcal{S}(G)=S(\lG,[\lG])$ and $\mathcal{R}(G)=R(\lG,[\lG])$, which we shall return to in the next section.\par
\begin{definition}
\label{param}
As seen in \ref{reflection}, for the collections $S(G_j,[G_j])$ and $R(G_i,[G_j])$, we have associated sets of cosets
\begin{align*}
C(S(G_j,[G_j]))&=\{\beta_1\operatorname{Hol}(G_i),\dots,\beta_s\operatorname{Hol}(G_i)\}\\
C(R(G_i,[G_j]))&=\{\alpha_1\operatorname{Hol}(G_i),\dots,\alpha_r \operatorname{Hol}(G_j)\}\\
\end{align*}
and we say $\{\beta_u\}$ parameterizes $S(G_j,[G_i])$ since $S(G_j,[G_i])=\{\beta_uG_i{\beta_u}^{-1}\}$ and $\{\alpha_v\}$ parameterizes $R(G_i,[G_j])$ since  $R(G_i,[G_j])=\{\alpha_vG_j{\alpha_v}^{-1}\}$.
\end{definition}
We also observe that any other set of coset representatives parameterizes the same set of subgroups. We should also point out that none of these sets of coset representatives are necessarily a complete set of coset representatives for some group containing $\operatorname{Hol}(G_i)$ or $\operatorname{Hol}(G_j)$.\par
If $|\operatorname{Aut}(G_i)|=|\operatorname{Aut}(G_j)|$ then $s=r$, and as we saw above, for a given $\beta_u$, $\Phi(\beta_u)=\beta_u^{-1}\in\alpha_v\operatorname{Hol}(G_j)$ for some $\alpha_v$. What we would like is to choose a set $\{\beta_1,\dots,\beta_s\}$ such that $\Phi(\{\beta_1,\dots,\beta_s\})$ parameterizes $R(G_i,[G_j])$. This is equivalent to finding a transversal of $C(S(G_j,[G_i]))$ which maps to a transversal of $C(R(G_i,[G_j]))$. The usual argument for this is some form of Hall's Marriage Theorem \cite{Hall35}, as is used to prove that for a subgroup of a group, one may find a set which is simultaneously a left and right transversal. However, again, the union of these cosets is not necessarily a group containing $\operatorname{Hol}(G)$ as a subgroup. Instead, we can more easily derive this from the theorem of K{\"o}nig in \cite{Konig} for bijections between sets partitioned into equal numbers of subsets.
\begin{lemma}
\label{pSpR}
If $|Aut(G_j)|=|Aut(G_j)|$ then it is possible to choose a set $P$ which parameterizes $S(G_j,[G_i])$ for which $\Phi(P)$ parameterizes $R(G_i,[G_j])$.
\end{lemma}
This result will be useful in the analysis of $\mathcal{S}(G)$ and $\mathcal{R}(G)$ to come later.
\section{Mutually Normalizing Subgroups}
For the case $G_j=G_i=G$ for $X=G$ for $G$ a given group, we consider
\begin{align*}
\mathcal{S}=\mathcal{S}(G)&=S(\lG,[\lG])\\
\mathcal{R}=\mathcal{R}(G)&=R(\lG,[\lG])\\
\end{align*}
as introduced in section 2. As seen in the proof of \ref{reflection}, the elements conjugating $\lG$ to a given member of $\mathcal{S}(G)$ or $\mathcal{R}(G)$ are {\it not} unique, but they {\it do} give rise to unique cosets of $\operatorname{Hol}(G)$. And since $G_i=G_j=\lG$ then $|Aut(G_i)|=|Aut(G_j)|$ obviously and so, $\mathcal{S}(G)$ is parameterized by $\{\beta_1,\dots,\beta_s\}$ and $\mathcal{R}(G)$ is parameterized by $\{\alpha_1,\dots,\alpha_s\}$ where, by \ref{pSpR}, we may order the $\{\beta_i\}$ and $\{\alpha_i\}$ so that $\Phi(\beta_i)\in \alpha_i\HG$.\par
By re-ordering $\{\beta_1,\dots,\beta_s\}$ as needed, the members of $\mathcal{S}(G)\cap\mathcal{R}(G)$ are therefore parameterized by a subset $\{\beta_1,\dots,\beta_{t}\}\subseteq \{\beta_1,\dots,\beta_s\}$ and as such $\{\beta_1^{-1},\dots,\beta_t^{-1}\}$ also parameterizes $\mathcal{S}(G)\cap \mathcal{R}(G)$. 
\begin{lemma}
\label{invclosure}
$\mathcal{S}(G)\cap\mathcal{R}(G)$ is parameterized by a set of left coset representatives $\hat{P}=\{\beta_1,\dots,\beta_t\}$ where, for some $\sigma\in S_t$ one has $\beta_i^{-1}\lG\beta_i=\beta_{\sigma(i)}\lG\beta_{\sigma(i)}^{-1}$. Moreover $\{\beta_1,\dots,\beta_t\}$ is also a set of right coset representatives, in that $\beta_1\HG\cup\dots\cup \beta_t\HG=\HG\beta_1\cup\dots\cup \HG\beta_t$.
\end{lemma}
\begin{proof}
By \ref{pSpR} there exists a set $P$ parameterizing $\mathcal{S}(G)$ such that $\Phi(P)$ parameterizes $\mathcal{R}(G)$. And since $\mathcal{S}(G)\cap\mathcal{R}(G)$ is a subset of both $\mathcal{S}(G)$ and $\mathcal{R}(G)$ obviously there exists $\hat{P}=\{\beta_1,\dots,\beta_t\}$ parameterizing $\mathcal{S}(G)\cap\mathcal{R}(G)$ which means $\Phi(\hat{P})=\{\beta_1^{-1},\dots,\beta_t^{-1}\}$ where
\begin{align*}
\mathcal{S}(G)\cap\mathcal{R}(G)&=\{\beta_1\lG\beta_1^{-1},\dots,\beta_t\lG\beta_t^{-1}\}\\
         &=\{\beta_1^{-1}\lG\beta_1,\dots,\beta_t^{-1}\lG\beta_t\}\\
\end{align*}
and so there is some $\sigma\in S_t$ such that$ \beta_i^{-1}\lG\beta_i^{-1}=\beta_{\sigma(i)}\lG\beta_{\sigma(i)}^{-1}$.\par
As to the second assertion, if we had $\HG\beta_i=\HG\beta_j$ then, in particular, $\beta_i=h\beta_j$ for some $h\in H$ which means $\beta_i^{-1}=\beta_j^{-1}h^{-1}$. This would imply therefore that $\beta_i^{-1}\lG\beta_i=(\beta_j^{-1}h^{-1})\lG(h\beta_j)=\beta_j^{-1}\lG\beta_j$, which implies $\beta_i^{-1}=\beta_j^{-1}$. We note that since $\beta_i\lG\beta_i^{-1}\in\mathcal{S}(G)\cap\mathcal{R}(G)$, the same is true for any $h\beta_i\in \HG\beta_i$.
\end{proof}
We again emphasize that this is not the simultaneous left/right transversal that exists for a subgroup of a given group, as a consequence of Hall's marriage theorem, but rather a consequence of the bijection $\Phi$ given above. Moreover, as a simultaneous left/right transversal, $\cup_{i=1}^{t} \beta_i\HG=\cup_{i=1}^{t}\HG\beta_i$ but one definitely cannot expect any of the left cosets to equal any right coset. In contrast however, $\cup_{i=1}^{t}\beta_i^{-1}\HG=\cup_{i=1}^{t}\beta_{\sigma(i)}\HG$ by virtue of $\beta_i^{-1}\HG=\beta_{\sigma(i)}\HG$ for each $i$. Also, as observed earlier, since $\lG\in\mathcal{S}(G)\cap\mathcal{R}(G)$ we may assume $id_B\in \{\beta_1,\dots,\beta_t\}$.\par
 One tantalizing possibility is whether there exists a set of parameters $\hat{P}=\{\beta_1,\dots,\beta_{t}\}$ parameterizing $\mathcal{S}(G)\cap\mathcal{R}(G)$ such that $\Phi(\hat{P})=\hat{P}$? Alas, the answer is no.\par
Our counterexample was found using the SmallGroups Library \cite{SmallGroups} and the group in question is the third group of order $40$ 
$$
G=\langle a,b\ |\ a^5=b^8=1;bab^{-1}=a^2\rangle
$$
which is a semidirect product $C_5\rtimes C_8$ given by the squaring map in $C_5$. For this group, $\mathcal{S}(G)\cap\mathcal{R}(G)$ has 14 members, parameterized by coset representatives $\{\beta_1,\dots,\beta_{14}\}$. To infer that {\it no} set of coset representatives of $\{\beta_1\HG,\dots,\beta_{14}\HG\}$ is closed under inversion, we begin by noting that for one of the non-trivial cosets $\beta \HG$ one has $\Phi(\beta\HG)=\beta \HG$. Without loss of generality, assume this is $\beta_1\HG$. We also find that this coset contains no elements of order 2. So in any set $\hat{P}$ of coset representatives, if $\widetilde{\beta}\in \hat{P}$ is that representing $\beta_1\HG$ then $\Phi(\hat{P})=\hat{P}$ would imply that $\Phi(\widetilde{\beta})=\widetilde{\beta}$, i.e. $\widetilde{\beta}^2=1$, which is impossible. We note, for later reference, that this $\beta_1 \HG$ which contains no elements of order $2$ is actually a coset in $T(G)=\NHG/\HG$. And we therefore infer that $\NHG$ is not a split extension of $\HG$.\par
So even though we cannot assume that a given set of coset representatives $\hat{P}=\{\beta_1,\dots,\beta_t\}$ parameterizing $\mathcal{S}(G)\cap\mathcal{R}(G)$ is closed under $\Phi$, \ref{invclosure} does give a type of inverse operation on some sets of parameters for $\mathcal{S}(G)\cap\mathcal{R}(G)$. Moreover, as we also observed, we may assume that such a parameter set contains $id_B$. However the counterexample just seen shows that there are cases where no $\{\beta_1,\dots,\beta_t\}$ parameterizing $\mathcal{S}(G)\cap\mathcal{R}(G)$ forms a group. So, in particular, $\mathcal{S}(G)\cap\mathcal{R}(G)$ is not the orbit of $\lG$ under a group of order equal to $|\mathcal{S}(G)\cap\mathcal{R}(G)|$.\par
\noindent We therefore ask a related question.
\begin{itemize}
\item Does $\cup_{i=1}^{t} \beta_i\operatorname{Hol}(G)$ form a group, and does this group act transitively on $\mathcal{S}(G)\cap\mathcal{R}(G)$?
\end{itemize}
In this way, we are looking for analogues of $\NHG$ and $T(G)$ which act on $\mathcal{H}(G)$. With the example of $\mathcal{H}(G)$ in mind, where $\NHG$ acts transitively and $T(G)=\NHG/\operatorname{Hol}(G)$ regularly, we again note that there need not exist any set of coset representatives which is itself a subgroup since $\NHG$ need not be a split extension of $\operatorname{Hol}(G)$. Moreover, if $\mathcal{S}(G)\cap\mathcal{R}(G)$ properly contains $\mathcal{H}(G)$ then such a group for which the orbit of $\lG$ is $\mathcal{S}(G)\cap\mathcal{R}(G)$ would contain $\NHG$ and therefore would not be a group extension of $\HG$.\par
Now, the set $\cup_{i=1}^{t} \beta_i\operatorname{Hol}(G)$ contains the identity since $\lG\in\mathcal{S}(G)\cap\mathcal{R}(G)$ but in order for this union of cosets to form a group, we need the following:
\begin{itemize}
\item For each $\beta_i h_1\in\beta_i\operatorname{Hol}(G)$,$\beta_j h_2\in\beta_j\operatorname{Hol}(G)$ there exists $\beta_k h_3\in\beta_k\operatorname{Hol}(G)$ such that $\beta_i h_1\beta_j h_2=\beta_kh_3$
\end{itemize}
Of course, basic group theory guarantees that if $\cup_{i=1}^{t} \beta_i\operatorname{Hol}(G)$ is closed in this way then it obviously contains inverses of all its elements. However, verifying this closure under multiplication is occasionally subtle. We note that, by \ref{invclosure}, we may choose a set of coset representatives $\{\beta_1,\dots,\beta_t\}$ such that
\begin{itemize}
\item if $\beta h\in \cup_{i=1}^{t} \beta_i\operatorname{Hol}(G)$ then $(\beta h)^{-1}\in\cup_{i=1}^{t} \beta_i\operatorname{Hol}(G)$
\end{itemize}
which, by itself, does not imply that $\cup_{i=1}^{t} \beta_i\operatorname{Hol}(G)$ is a group of course. Also by \ref{invclosure}, since this set of left coset representatives is also a set of right coset representatives, the closure of  $\cup_{i=1}^{t} \beta_i\operatorname{Hol}(G)$ is equivalent to the statement that for $\beta_i,\beta_j$, there exists $\beta_k$ such that $\beta_i\beta_j\lG\beta_j^{-1}\beta_i^{-1}=\beta_k\lG\beta_k^{-1}$. \par
\noindent However, there is an essential issue that must be addressed, namely that, in fact, it need not be the case that $\beta_i\beta_j\lG\beta_j^{-1}\beta_i^{-1}\in\mathcal{S}(G)\cap\mathcal{R}(G)$. Here is an example. Let $G=C_4\times C_2$ which can be represented as the regular permutation group $\langle (1,2)(3,4)(5,6)(7,8), (1,3,5,7)(2,4,6,8)\rangle\leq S_8$. One may compute $\mathcal{S}(G)\cap\mathcal{R}(G)$ which has 8 members, and one may compute a set of coset representatives which parameterize these 8 groups
$$
\{(), (2,5)(4,7),(5,6)(7,8), (4,8), (2,3)(6,7), (2,4)(3,7)(6,8), (2,3)(4,8)(6,7), (2,4)(6,8)\}
$$
but one may show that $(4,8)(2,5)(4,7)=(2,5)(4,7,8)$ conjugates $G$ to a group that is {\it not} normalized by $\lG$, ergo not in $\mathcal{R}(G)$ so obviously not in $\mathcal{S}(G)\cap\mathcal{R}(G)$. (Note, the regular permutation group is the group $G$ given above as opposed to the left regular representation of an abstract group.) This counterexample would seem to render this whole enterprise pointless, but instead, we shall examine {\it why} this closure property fails, and more importantly, what is necessary for it to hold. \par
In particular we shall look more broadly at subsets $\mathcal{X}(G)\subseteq\mathcal{S}(G)\cap\mathcal{R}(G)$ (containing $\lG$) and consider the coset representatives of $\HG$, $\{\beta_1,\dots,\beta_m\}$, which parameterize $\mathcal{X}(G)$, and look at necessary and sufficient conditions for the resulting union of cosets to form a group (containing $\HG$) for which the orbit of $\lG$ under this group is $\mathcal{X}(G)$.\par
\noindent {\bf Notation:} Going forward, if $\mathcal{X}(G)\subseteq\mathcal{S}(G)\cap\mathcal{R}(G)$, (where we assume $\lG\in\mathcal{X}(G)$) we shall use the notation $\pi(\mathcal{X}(G))$ to denote a set of coset representatives with respect to $\HG$ which conjugate $\lG$ to the elements of $\mathcal{X}(G)$. Whereas the choice of coset representatives is not unique, the cosets of $\HG$ are, of course, uniquely determined by $\mathcal{X}(G)$.\par
So for $\pi(\mathcal{X}(G))=\{\beta_1,\dots,\beta_m\}$ we wish to consider conditions which make $\cup_{i=1}^m\beta_i\HG$ a group under which the orbit of $\lG$ is $\mathcal{X}(G)$. One thing to note is that in the passage from $\mathcal{S}(G)\cap\mathcal{R}(G)$ to a given $\mathcal{X}(G)\subseteq \mathcal{S}(G)\cap\mathcal{R}(G)$ it is not altogether obvious whether there exists a $\pi(\mathcal{X}(G))$ for which a version of \ref{invclosure} holds. That is, whether the map $\Phi$ restricts to a bijection of $\cup_{i=1}^m\beta_i\HG$ to itself. \par
In order to streamline the notation and the subsequent analysis, we define the following.
\begin{definition}
A subset $A\subseteq B=Perm(G)$ is 'conj-closed' if for each $\alpha,\beta\in A$ there exists $\gamma\in A$ such that $\alpha\beta\lG\beta^{-1}\alpha^{-1}=\gamma\lG\gamma^{-1}$. If so then we can define a binary operation $*$ on $A$ by $\alpha*\beta=\gamma$.\par
And we say $A$ is 'inv-closed' if for each $\beta\in A$, there exists an $\alpha\in A$ such that $\beta^{-1}\lG\beta=\alpha\lG\alpha^{-1}$.
\end{definition}
One immediate observation to make is that by \ref{invclosure} there exists a $\pi(\mathcal{S}(G)\cap\mathcal{R}(G))$ that is inv-closed. And in general, for $\mathcal{X}(G)$ (containing $\lG$) parameterized by $\pi(\mathcal{X}(G))=\{\beta_1,\dots,\beta_m\}$, one has that $\cup_{i=1}^m\beta_i\HG$ is closed under multiplication (and therefore a group) if and only if $\pi(\mathcal{X}(G))$ is conj-closed. Of course, since $G$ is finite, conj-closure implies inv-closure. What the next result shows is a necessary condition on $\mathcal{X}(G)\subseteq\mathcal{S}(G)\cap\mathcal{R}(G)$ (where again $\lG\in\mathcal{X}(G)$) in order that a given $\pi(\mathcal{X}(G))=\{\beta_1,\dots,\beta_m\}$ be conj-closed, i.e. that $\cup_{i=1}^{t}\beta_iHol(G)$ be a group.
\begin{proposition}
\label{conjclosenorm}
If $\mathcal{X}(G)\subseteq\mathcal{S}(G)\cap\mathcal{R}(G)$ where $\pi(\mathcal{X}(G))=\{\beta_1,\dots,\beta_m\}$ and $\pi(\mathcal{X}(G))$ is conj-closed then all $N\in\mathcal{X}(G)$ normalize each other.
\end{proposition}
\begin{proof}
If for each $\beta_i,\beta_j$ there is a $\beta_k$ such that $\beta_i\beta_j\lG\beta_j^{-1}\beta_i^{-1}=\beta_k\lG\beta_k^{-1}$ then, in particular, if $|\beta_i|=a$ then $\beta_i^{a-1}\lG\beta_i^{-(a-1)}=\beta_i^{-1}\lG\beta_i=\beta_k\lG\beta_k^{-1}$ namely that $\beta_i^{-1}\operatorname{Norm_B}(G)=\beta_k\operatorname{Norm_B}(G)$. As such, for each $\beta_i$ we have that $\conj[\lG][\beta_i]$ and $\xconj[\lG][\beta_i]$ lie in $\mathcal{X}(G)$.\par
So we have that all conjugates of $\lG$ with respect to products, inverses, and combinations thereof from $\pi(\mathcal{X}(G))$ lie in $\mathcal{X}(G)$. So $\conj[\lG][\beta_i]\leq\operatorname{Norm_B}(\conj[\lG][\beta_j])$ if and only if
\begin{align*}
\beta_j^{-1}\beta_i\lG\beta_i^{-1}\beta_j&\leq Norm_B(\lG)=\operatorname{Hol}(G)\\
\lG&\leq\operatorname{Norm_B}(\beta_i^{-1}\beta_j\lG\beta_j^{-1}\beta_i)\\
\end{align*}
and symmetrically, $\conj[\lG][\beta_j]\leq\operatorname{Norm_B}(\conj[\lG][\beta_i])$ if and only if 
\begin{align*}
\beta_i\beta_j^{-1}\lG\beta_j\beta_i^{-1}&\leq\operatorname{Hol}(G)\\
\lG&\leq Norm_B(\beta_j^{-1}\beta_i\lG\beta_i^{-1}\beta_j).
\end{align*} Therefore, all elements in $\mathcal{X}(G)$ normalize each other.
\end{proof}
The point is that, if the elements of $\mathcal{X}(G)$ fail to mutually normalize each other, then it would contradict the closure property $\beta_i\beta_j\lG\beta_j^{-1}\beta_i^{-1}=\beta_k\lG\beta_k^{-1}$. The takeaway from this is that if any two elements of $\mathcal{X}(G)$ fail to normalize each other, then $\cup_{i=1}^{m}\beta_i\operatorname{Hol}(G)$ is not a group which is equivalent to having $\beta_i$,$\beta_j\in\pi(\mathcal{X}(G))$ so that $\beta_i\beta_j\lG\beta_j^{-1}\beta_i^{-1}\not\in\mathcal{X}(G)$. This is indeed the case for the counterexample we examined above, in particular $(2,5)(4,7)G(4,7)(2,5)$ normalizes $(4,8)G(4,8)$ but not the reverse.\par
This raises the question of whether the converse of the above result holds. A close examination of the proof reveals that we were able to derive that for each $\beta\in\pi(\mathcal{X}(G))$ one has that $\beta^{-1}\lG\beta\in\mathcal{X}(G)$. And this, in turn, implied that the elements of $\mathcal{X}(G)$ normalized each other. We note a partial converse of the above result when all the groups in $\mathcal{S}(G)\cap\mathcal{R}(G)$ normalize each other.
\begin{proposition}
\label{SRQ}
Let $\pi(\mathcal{S}(G)\cap\mathcal{R}(G))$ parameterize $\mathcal{S}(G)\cap\mathcal{R}(G)$. If the elements of $\mathcal{S}(G)\cap\mathcal{R}(G)$ are mutually normalizing, and if $\pi(\mathcal{S}(G)\cap\mathcal{R}(G))$ is inv-closed then $\pi(\mathcal{S}(G)\cap\mathcal{R}(G))$ is conj-closed.
\end{proposition}
\begin{proof}
For $\pi(\mathcal{S}(G)\cap\mathcal{R}(G))=\{\beta_1,\dots,\beta_t\}$ being inv-closed we have that for each $\beta_i$, one has $\beta_i\lG\beta_i^{-1}\in\mathcal{S}(G)\cap\mathcal{R}(G)$ if and only if $\beta_i^{-1}\lG\beta_i\in\mathcal{S}(G)\cap\mathcal{R}(G)$. And $\mathcal{S}(G)\cap\mathcal{R}(G)$ being mutually normalizing implies that for all $\beta_i,\beta_j$ we have $\beta_i\lG\beta_i^{-1}\leq Hol(\beta_j\lG\beta_j^{-1})$. So for a given pair $\beta_i,\beta_j$ we find that $\beta_i\beta_j\lG\beta_j^{-1}\beta_i^{-1}\in \mathcal{S}(G)\cap\mathcal{R}(G))$ if and only if 
\begin{align*}
\lG &\leq Hol(\beta_i\beta_j\lG\beta_j^{-1}\beta_i^{-1})\\
\beta_i\beta_j\lG\beta_j^{-1}\beta_i^{-1}&\leq Hol(\lG)\text{ i.e. $Hol(G)$}
\end{align*}
but the first condition is equivalent to $\beta_i^{-1}\lG\beta_i\leq Hol(\beta_j\lG\beta_j^{-1})$ which holds by inv-closure $\beta_i^{-1}\lG\beta_i\in\mathcal{S}(G)\cap\mathcal{R}(G)$ and the fact that $\mathcal{S}(G)\cap\mathcal{R}(G)$ is mutually normalizing. A similar argument shows that $\beta_i\beta_j\lG\beta_j^{-1}\beta_i^{-1}\leq Hol(\lG)$ holds as well.
\end{proof}
In the above proof, we cannot replace $\mathcal{S}(G)\cap\mathcal{R}(G)$ with an arbitrary mutually normalizing subset $\mathcal{X}(G)$ with an inv-closed $\pi(\mathcal{X}(G))$. The reason for this is that for a given $\beta_i$, $\beta_j$ in $\mathcal{X}(G)$, our proof above yields that $\beta_i\beta_j\lG\beta_j^{-1}\beta_i^{-1}$ is an element of $\mathcal{S}(G)\cap\mathcal{R}(G)$ but not necessarily $\mathcal{X}(G)$. Besides this, we have a straightforward counterexample. If we go back to the order 40 group example one more time, we may choose $\mathcal{X}(G)$ to be any 3 element subset of $\mathcal{H}(G)$ (which includes $\lG$) where the coset representatives parameterizing the other (non-$\lG$) elements of $\mathcal{X}(G)$ have order 2. All of these are mutually normalizing and $\pi(\mathcal{X}(G))$ is certainly inv-closed. However, $\pi(\mathcal{X}(G))$ cannot be conj-closed as it would imply that $T(G)$ (of order 4) has a subgroup of order $3$.\par
We have the following summary.
\begin{proposition}
\label{ScapRloop}
Let $\mathcal{X}(G)\subseteq \mathcal{S}(G)\cap \mathcal{R}(G)$ (containing $\lG$) be parameterized by $\pi(\mathcal{X}(G))=\{\beta_1,\dots,\beta_m\}$, where $\pi(\mathcal{X}(G))$ is conj-closed and let $\XHG=\cup_{i=1}^{m}\beta_i \operatorname{Hol}(G)$. The following properties hold:
\begin{description}
\item[(a)] All the groups in $\mathcal{X}(G)$ normalize each other.\par
\item[(b)] $Orb_{\XHG}(\lG)=\mathcal{X}(G)$
\item[(c)] $|\XHG|=|\mathcal{X}(G)|\cdot|\operatorname{Hol}(G)|$ 
\end{description}
\end{proposition}
\begin{proof}
We have seen that condition (a) is consequence of conj-closure already, and conditions (b) and (c) are the orbit stabilizer theorem, where $Hol(G)$ is the stabilizer subgroup of $\lG$. 
\end{proof}
Another way to view the mutually normalizing property is as follows. If we think of the groups in $\mathcal{S}(G)\cap\mathcal{R}(G)$ as the nodes in a directed graph, where there is an 'arrow' $N_1\rightarrow N_2$ if $N_1$ normalizes $N_2$ then there are arrows $\lG\rightarrow N$ for all $N\in\mathcal{S}(G)\cap\mathcal{R}(G)$ and if two groups $N_1,N_2\in\mathcal{S}(G)\cap\mathcal{R}(G)$ mutually normalize each other then we may view $N_1$ and $N_2$ as being connected by simply an 'edge' $N_1\leftrightarrow N_2$. So any $\mathcal{X}(G)$ satisfying \ref{ScapRloop} would be a clique (complete subgraph) within this digraph.\par
We observe that $\mathcal{X}(G)=\mathcal{H}(G)$ clearly satisfies the conditions of \ref{ScapRloop} and is a prototype example. Indeed, for $\mathcal{X}(G)=\mathcal{H}(G)$ one has that $\XHG=\NHG$. As our goal is to generalize the multiple holomorph, we will consider a very specific subset of $\mathcal{S}(G)\cap\mathcal{R}(G)$ which will contain $\mathcal{H}(G)$ which will be defined for all groups $G$, and will satisfy the necessary property of being mutually normalizing. And as this is a set that is defined universally for any group $G$, we will use a different notation than $\mathcal{X}(G)$.\par
\begin{definition}
Let $\mathcal{Q}(G)=\underset{N\in\mathcal{S}(G)\cap\mathcal{R}(G)}{\bigcap} \{M\in \mathcal{S}(G)\cap\mathcal{R}(G)\ |\ N\text{ normalizes }M\}$.
\end{definition}
Another way to view this is as follows. For a given $M\in\mathcal{S}(G)\cap\mathcal{R}(G)$, since $M$ is regular, it makes sense to define $$\mathcal{S}(M)=\{\text{regular }A\leq \operatorname{Norm}_B(M)\ |\ A\cong M\}$$ which is the direct analogue of $\mathcal{S}(G)=\mathcal{S}(\lG)=S(\lG,[\lG])$ so here, $\mathcal{S}(M)=S(M,[M])$, and so
$$
\mathcal{Q}(G)=\underset{N\in\mathcal{S}(G)\cap\mathcal{R}(G)}{\bigcap}\{M\in\mathcal{S}(G)\cap\mathcal{R}(G)\ |\ N\in\mathcal{S}(M)\}
$$
which is a subset of $\mathcal{S}(G)\cap\mathcal{R}(G)$ which we shall show has the properties we want. The first most important observation is the following (nearly obvious) fact.
\begin{lemma}
\label{Qmutual}
The members of $\mathcal{Q}(G)$ mutually normalize each other.
\end{lemma}
\begin{proof}
If $A,B\in\mathcal{Q}(G)$ then for all $N\in\mathcal{S}(G)\cap\mathcal{R}(G)$ one has that $N$ normalizes $A$ and $N$ normalizes $B$, so in particular, if $N=A$ then we have that $A$ normalizes $B$, and symmetrically, for $N=B$ we have that $B$ normalizes $A$.
\end{proof}
And as we wish to have a generalization of the multiple holomorph, we observe the following.
\begin{lemma}
\label{HinQ}
For $\mathcal{Q}(G)$ defined above, one has $\mathcal{H}(G)\subseteq\mathcal{Q}(G)$.
\end{lemma}
\begin{proof}
If $M\in\mathcal{H}(G)$ then $\operatorname{Norm}_B(M)=\operatorname{Norm}_B(\lG)=\operatorname{Hol}(G)$, so for any $N\in\mathcal{S}(G)\cap\mathcal{R}(G)$ one has $N\leq \operatorname{Hol}(G)=\operatorname{Norm}_B(M)$ for all $N\in\mathcal{S}(G)\cap\mathcal{R}(G)$. That is, $N\in\mathcal{S}(M)$ for all $N\in\mathcal{S}(G)\cap\mathcal{R}(G)$.
\end{proof}
%
The set $\mathcal{H}(G)$ is maximal with respect to the property of having the same normalizer as $\lG$. This prompts the question of whether $\mathcal{Q}(G)$ is maximal with respect to the property of being mutually normalizing. For abelian groups there are counterexamples. If $G=C_2\times C_4$ then $\mathcal{H}(G)$ contains two members, and using GAP \cite{GAP4}, one may show that $|\mathcal{S}(G)\cap\mathcal{R}(G)|$ has 8 members, and that $\mathcal{Q}(G)=\mathcal{H}(G)$. However, there are subsets of $\mathcal{S}(G)\cap\mathcal{R}(G)$ (properly containing $\mathcal{Q}(G)$) which are mutually normalizing. For non-abelian groups, the computational evidence suggests that $\mathcal{Q}(G)$ {\it is} the maximal subset of $\mathcal{S}(G)\cap\mathcal{R}(G)$ that contains $\mathcal{H}(G)$ and is mutually normalizing.\par
One further observation that we can make is about the size of $\mathcal{Q}(G)$ compared to that of $\mathcal{H}(G)$.
\begin{proposition}
For any group $G$, one has $|\mathcal{H}(G)| \big\| |\mathcal{Q}(G)|$.
\end{proposition}
\begin{proof}
If $M\in\mathcal{Q}(G)$ then $M=\gamma\lG\gamma^{-1}$ for some $\gamma$. Now, if $\tau\in T(G)$ then $\HG=\tau\HG\tau^{-1}=\tau\operatorname{Norm}_B(\lG)\tau^{-1}=\operatorname{Norm}_{B}(\tau\lG\tau^{-1})$ and so 
\begin{align*}
\operatorname{Norm}_B(M)&=\operatorname{Norm}_B(\gamma\lG\gamma^{-1})=\gamma\operatorname{Norm}_B(\lG)\gamma^{-1}\\
&=\gamma\operatorname{Norm}_B(\tau\lG\tau^{-1})\gamma^{-1}=\operatorname{Norm}_B(\gamma\tau\lG\tau^{-1}\gamma^{-1})
\end{align*}
which means that $\mathcal{H}(M)=\{\gamma\tau\lG\tau^{-1}\gamma^{-1}\ |\ \tau\in T(G)\}$.\par
So $M=\gamma\lG\gamma^{-1}\in\mathcal{Q}(G)$ means that $\beta\lG\beta\leq\operatorname{Norm}_B(\gamma\lG\gamma^{-1})$ for all $\beta\in\pi(\mathcal{S}(G)\cap\mathcal{R}(G))$, which is equivalent to 
$$
\gamma^{-1}\beta\lG\beta^{-1}\gamma\leq Hol(G)
$$
for all $\beta\in \pi(\mathcal{S}(G)\cap\mathcal{R}(G))$. So we can ask whether $M\in\mathcal{Q}(G)$ implies $\mathcal{H}(M)\subseteq\mathcal{Q}(G)$. But this is equivalent to 
$$
\beta\lG\beta^{-1}\leq\operatorname{Norm}_B(\gamma\tau\lG\tau^{-1}\gamma^{-1})\text{\ \ \ {\bf (*)}}
$$
for all $\beta\in \pi(\mathcal{S}(G)\cap\mathcal{R}(G))$ and all $\tau\in T(G)$, which is equivalent to $\tau^{-1}\gamma^{-1}\beta \lG\beta^{-1}\gamma\tau\leq\operatorname{Norm}_B(\lG)=\HG$, but this is equivalent to $\gamma^{-1}\beta \lG\beta^{-1}\gamma\leq \tau\HG\tau^{-1}$, but $\tau\HG\tau^{-1}=\HG$ of course so {\bf (*)} does hold. So since $M\in\mathcal{Q}(G)$ implies $\mathcal{H}(M)\subseteq\mathcal{Q}(G)$ then $\mathcal{Q}(G)$ can be partitioned into distinct classes, all of which have the same holomorph (normalizer), and the result follows.
\end{proof}
And as one would expect, like LaGrange's theorem, divisibility results like this have many consequences, such as this one.
\begin{corollary}
If $G$ is non-abelian and $|\mathcal{S}(G)\cap\mathcal{R}(G)|$ is odd then $\mathcal{Q}(G)$ is a proper subset of $\mathcal{S}(G)\cap\mathcal{R}(G)$.
\end{corollary}
There is one other implication of \ref{ScapRloop} to mention which is fairly interesting. Recall that we defined $\mathcal{S}(G)=S(\lG,[\lG])$, the collection of regular subgroups of $\operatorname{Norm}_B(\lG)$ that are isomorphic to $\lG$, and $\mathcal{R}(G)=R(\lG,[\lG])$ the collection of regular subgroups of $B$, isomorphic to $\lG$ and normalized by $\lG$. For a given $G$ with associated $\mathcal{Q}(G)=\mathcal{S}(G)\cap\mathcal{R}(G)$, since each $N\in\mathcal{Q}(G)$ is regular, and isomorphic to $\lG$ then it makes sense to consider $\mathcal{Q}(N)=S(N,[N])\cap R(N,[N])$, namely the regular subgroups $N$ of $\operatorname{Norm}_B(N)$, isomorphic to $N$, that are normalized by $N$. With this in mind, we have the following consequence of \ref{ScapRloop}.\par
\begin{proposition}
\label{equalQ}
If $G$ is a group and $N\in\mathcal{Q}(G)$, one has $\mathcal{Q}(N)=\mathcal{Q}(G)$.
\end{proposition}
\begin{proof}
This follows directly from \ref{ScapRloop} since if $N\in\mathcal{Q}(G)$ then $N$ normalizes, and is normalized by, all other elements of $\mathcal{Q}$. But since all such $N$ are isomorphic regular subgroups then $\operatorname{Norm}_B(N)\cong \operatorname{Norm}_B(\lG)$ so $|\mathcal{Q}(N)|=|\mathcal{Q}(G)|$, so $\mathcal{Q}(N)=\mathcal{Q}(G)$.
\end{proof}
We now define the following.
\begin{definition}
For any $\pi(\mathcal{Q}(G))=\{\beta_1,\dots,\beta_m\}$ which parameterizes $\mathcal{Q}(G)$ let the {\it quasi-holomorph} of $G$ be $\QHG=\cup_{i=1}^{m}\beta_i \operatorname{Hol}(G)$.
\end{definition}
Before going any further, there is a potential obstruction to this whole development which must be dealt with right now. Specifically we need to determine whether $\QHG$ is closed, namely is a group, i.e. that a $\pi(\mathcal{Q}(G))$ exists which satisfies \ref{ScapRloop}. If $\mathcal{Q}(G)=\mathcal{S}(G)\cap\mathcal{R}(G)$ or $\mathcal{Q}(G)=\mathcal{H}(G)$ then this is automatic by \ref{SRQ}. But, as we shall see, there are $G$ for which each of the containments 
$$
\mathcal{H}(G)\subseteq\mathcal{Q}(G)\subseteq\mathcal{S}(G)\cap\mathcal{R}(G)
$$
is proper. It is not clear to the author that in this case, there necessarily exists a $\pi(\mathcal{Q}(G))$ which is conj-closed, i.e. that $\QHG$ must be a group. It is our conjecture, however, that $\QHG$ {\it is} a group for all $G$, but at the present time no proof is apparent. We shall show that $\QHG$ is a group for cyclic and dihedral $G$, and in the final section we shall exhibit calculations exhibiting the sizes of the sets $\mathcal{S}(G)\cap\mathcal{R}(G)$, $\mathcal{H}(G)$, and $\mathcal{Q}(G)$ and that for all those explored, $\QHG$ is a group.
\par
\subsection{Loops and Zappa-Sz\'ep Extensions}
The usage of the term 'quasi' is meant to be suggestive of the notions of quasigroup and loop. We remind the reader of some basic definitions, but the terminology in this area is far from consistent throughout the literature, although a standard reference for these systems is \cite{Bruck} for example.
\begin{definition}
A set with a binary operation $(T,*)$ is a left (resp. right) {\it quasigroup} if the equation $a*x=b$ (resp. $y*a=b$) has a unique solution for every $a,b\in T$. And $(T,*)$ is a quasigroup if it is simultaneously a left and right quasigroup.\par
If, in addition, a left (resp. right) quasigroup $(T,*)$ has right (resp. left) identity, namely an element $1\in T$ such that $a*1=a$ (resp. $1*a=a$) for all $a\in T$ then it is a left (resp. right) {\it loop}. And it is a loop if it is quasigroup with a simultaneous left and right identity.
\end{definition}
Note, the pair $(T,*)$ being a quasigroup (two sided) is equivalent to the Cayley table for it being a Latin square, and if it has an identity then it contains a column and row which are identically ordered. Unlike a group, however, a quasigroup/loop need not be associative, but an associative loop is a group.\par
One of the simplest examples of a left quasigroup $(T,*)$ is when $T$ is a left transversal for a subgroup $H\leq G$ where $a*b=c$ if $ab\in cH$. And if $h\in T\cap H$ then clearly $a*h=a$, so, in fact, $(T,*)$ is a left loop. What one derives from \ref{ScapRloop} is that one may define the binary operation on $\pi(\mathcal{Q}(G))$ in terms of $\mathcal{Q}(G)$. Specifically, if one defines $\beta_i*\beta_j=\beta_k$ if $\beta_i\beta_j\lG\beta_j^{-1}\beta_i^{-1}=\beta_k\lG\beta_k^{-1}$ then this is the same as the transversal structure. In addition, we pointed out that any $\pi(\mathcal{Q}(G))$ contains exactly one element of $\operatorname{Hol}(G)$, which corresponds to $\lG\in\mathcal{Q}(G)$, and this element acts as a right identity. As such, if $\pi(\mathcal{Q}(G)$ is conj-closed then, with respect to this '$*$' operation, it is a loop.\par
We saw above that $\mathcal{H}(G)$ is the orbit of $\lG$ under the action of $\NHG$, where, since $\operatorname{Hol}(G)$ is the stabilizer subgroup for this action, $|\mathcal{H}(G)|=|T(G)|=|\NHG/\operatorname{Hol}(G)|$. For $\mathcal{H}(G)$ one has that $\NHG$ is an extension of $\operatorname{Hol}(G)$, namely
$$
\operatorname{Hol}(G)\hookrightarrow \NHG\twoheadrightarrow T(G)
$$
and indeed it is the maximal subgroup of $Perm(G)$ which is an extension of $\operatorname{Hol}(G)$. As $\mathcal{H}(G)\subseteq\mathcal{Q}(G)$ the quasi-holomorph $\QHG$ is potentially {\it larger} than $\NHG$ and therefore not an extension of $\operatorname{Hol}(G)$ in the usual sense. As we observed earlier, when $\mathcal{Q}(G)=\mathcal{S}(G)\cap\mathcal{R}(G)$ we may assume that $\pi(\mathcal{Q}(G))$ contains the identity and is closed under the taking of inverses. This is suggestive of the possibility that $\QHG$ is a Zappa-Sz\'ep extension of $\operatorname{Hol}(G)$, namely that $\QHG=\pi(\mathcal{Q}(G))\operatorname{Hol}(G)$ an internal product (also called the knit-product or Zappa-Sz\'ep product) of $\operatorname{Hol}(G)$ and some $\pi(\mathcal{Q}(G))$ if $\pi(\mathcal{Q}(G))$ happens to form a subgroup of $\QHG$.
\par
The term 'Zappa-Sz\'ep Extension' appears in \cite{Araujo2011} for the category of semi-groups, and the view of such internal products as extensions analogous to traditional split extensions appears in \cite{Ates2009}. Early references to such internal products (of subgroups of a given group) appears, for example, in Sz\'ep's original paper \cite{Szep1950} on the subject, but was studied by others as well, around the same time.\par
Looking back to the case of $\mathcal{H}(G)$ and $\NHG$ for a moment, we note that, although $\NHG$ is an extension of $\operatorname{Hol}(G)$, it need not necessarily be a split extension in that $\NHG=\operatorname{Hol}(G)M$ for some $M\leq\NHG$ where $Orb_{M}(\lG)=\mathcal{H}(G)$ and $M\cap \operatorname{Hol}(G)=\{id_B\}$. Indeed, using GAP \cite{GAP4} and the SmallGroups library \cite{SmallGroups} we can exhibit specific examples. For two groups of order 40, specifically $G\cong \langle a,b\ |\ a^5=b^8=1,bab^{-1}=a^2\rangle$ and $G\cong \operatorname{Hol}(C_{5})\times C_2$, $T(G)$ is isomorphic to the Klein-4 group, but for both, $\NHG$ contains no subgroups $M$ isomorphic to $V$ such that $M$ parameterizes $\mathcal{H}(G)$ and $\NHG=\operatorname{Hol}(G)M$.\par
However, upon examination of other low order groups, it seems that it's generally the case that $\NHG$ {\it is} a split extension of $\operatorname{Hol}(G)$. For example, in \cite[Thm 2.11]{Kohl-Multiple-2014}, for the dihedral groups $D_{n}=\langle x,t\ |\ x^n=1,t^2=1,xt=tx^{-1}\rangle$ we have:
\begin{theorem}\cite[2.11]{Kohl-Multiple-2014} 
For the $n$-th dihedral group $D_n$, of order $2n$ we have:
$$
\mathcal{H}(D_n)=\{\langle \rho(x)\phi_{u+1,1},\rho(t)\phi_{(0,u)}\rangle\ |\ u\in\Upsilon_n\}
$$
where $T(D_n)\cong \Upsilon_n$, and $\phi_{(i,j)}(t^ax^b)=t^ax^{ia+jb}\in \operatorname{Aut}(D_n)$.
\end{theorem}
 As $\QHG$ is not an extension in the usual sense, split or otherwise, the analog is precisely the representation of $\QHG$ as a Zappa-Sz\'ep extension of $\operatorname{Hol}(G)$, if possible. When $\NHG$ is split, and $\QHG$ is a Zappa-Sz\'ep product, then we have a natural consequence of \ref{SRcontainsH} which, even if there are many different isomorphism classes of $\pi(\mathcal{Q}(G))$, places some restrictions. 
\begin{proposition}
\label{TandQ}
If $\NHG\cong \operatorname{Hol}(G)\rtimes M$ is a split extension and $\QHG=\pi(\mathcal{Q}(G))\operatorname{Hol}(G)$ is a Zappa-Sz\'ep product for a $\pi(\mathcal{Q}(G))$ which is a group, then $\pi(\mathcal{Q}(G))$ contains a subgroup isomorphic to $M$.
\end{proposition}
\begin{proof}
If $\pi(\mathcal{Q}(G))$ is a subgroup of $\QHG$ then, since it parameterizes $\mathcal{Q}(G)$ then, by regularity of the action on $\mathcal{Q}(G)$, it intersects $\operatorname{Hol}(G)$ trivially and contains a unique subset $\widetilde M$ which parameterizes $\mathcal{H}(G)$. Since $\pi(\mathcal{Q}(G))$ is a group, then $1\in\widetilde M$ and if $\beta\widetilde M$ parameterizes $N\in\mathcal{H}(G)$ then so does $\beta^{-1}$ and so $\beta^{-1}h\in\widetilde M$ for some $h\in \operatorname{Hol}(G)$ but then $\beta\beta^{-1}h=h\in\pi(\mathcal{Q}(G))$ which implies $h=1$. So the only check that needs to be made is that $\widetilde M$ is closed. However, this follows from \ref{ScapRloop} since if $\alpha,\beta\in\widetilde M$ then by definition of $\mathcal{H}(G)$, $\alpha\beta=\gamma h$ for some $\gamma\in\widetilde M$ and $h\in H$. However this would imply $\gamma^{-1}\alpha\beta=h\in\pi(\mathcal{Q}(G))$ implying that, again, $h=1$. Thus $\widetilde M\cong M$ where, by assumption, $\NHG$ contains a subgroup isomorphic to $M$. Note that $M\cong T(G)$ of course.
\end{proof}
As mentioned earlier, Miller \cite{Miller1908} showed that for $G=C_{p^n}$ for $p$ an odd prime, one has that $\mathcal{H}(G)=\{\lG\}$ only, i.e. $T(G)$ is trivial. As we shall see in the next section, the situation is quite different for $\mathcal{Q}(C_{p^n})$ and $Q\!\operatorname{Hol}(C_{p^n})$. 
\section{Mutually Normalizing Subgroups of $\operatorname{Hol}(C_{p^n})$}
The subgroups of the holomorph of a cyclic group of odd prime power order were studied in \cite{Kohl1998} as part of the enumeration of the Hopf-Galois structures on radical extensions of degree $p^n$. We will utilize some of the technical information therein, in particular to establish \ref{gen-cyc-SR} and \ref{cyc-mutual}. We shall also (temporarily) use the notational conventions in \cite{Kohl1998}.\par 
\noindent Let $\langle\sigma\rangle$ be the cyclic group of order $p^n$ and let $\langle\delta\rangle$ be its automorphism group where $\sigma^{\delta}=\sigma^{\pi}$ where $\pi$ is a primitive root mod $p^n$. As such, $\operatorname{Hol}(\langle\sigma\rangle)$ consists of order pairs of the form $(\sigma^r,\delta^s)$ where, of course, $(\sigma^{r_1},\delta^{s_1})(\sigma^{r_2},\delta^{s_2})=(\sigma^{r_1+\pi^{s_1}r_2},\delta^{s_1+s_2})$. 
\begin{lemma}
\label{gen-cyc-SR}
For $\langle\sigma\rangle$ as given above, an element of $\mathcal{S}\cap\mathcal{R}$ is generated by $(\sigma^i,\delta^{p^k(p-1)})$ for $i\in U_{p^{n-1-k}}$ where $n\leq 2k+2$ i.e. $k=m\!-\!1..n\!-\!1$ where $m=[\frac{n}{2}]$.
\end{lemma}
\begin{proof}A consequence of Lemma 3.2 in \cite{Kohl1998}.
\end{proof}
\noindent As such, we have the following.\par\par
\begin{proposition} 
\label{cyc-mutual}
For $G=C_{p^n}$ we have that $\mathcal{Q}(G)=\mathcal{S}(G)\cap\mathcal{R}(G)$.
\end{proposition}
\begin{proof}
Direct calculation using the above lemma.
\end{proof}
For $G\cong C_{2^n}$ we have $\operatorname{Aut}(G)\cong\mathbb{Z}_{2^n}^{*}$ which, of course, is non-cyclic, which contrasts with the odd prime power case studied above. However, it turns out that the enumeration is nearly identical. In \cite{Byott2007}, Byott determined $\mathcal{S}(G)$ which we give here.
\begin{lemma}\cite[Lemma 7.3]{Byott2007} 
Let $G=\langle\sigma\rangle$ by cyclic of order $2^n$ and for $s\in\mathbb{Z}_{2^n}^{*}$ let $\delta_s\in \operatorname{Aut}(G)$ be that automorphism such that $\delta_s\sigma\delta_s^{-1}=\sigma^s$. Then $\operatorname{Hol}(G)$ contains exactly $2^{n-2}$ cyclic subgroups of order $2^n$, namely $\langle(\sigma,\delta_s)\rangle$ for each $s$ where $s\equiv 1(mod\ 4)$.
\end{lemma}
Now, for $n=1,2$ the only cyclic subgroups of order $2^n$ of $\operatorname{Hol}(G)$ are $G=\langle\sigma\rangle$ itself. And for $n\geq 3$ the set of those $s\in U_{2^n}$ for which $s\equiv 1(mod\ 4)$ lie in the cyclic subgroup generated by $5$, which has order $2^{n-2}$. That is, $\mathcal{S}(G)=\{N_s\ |\ \ s\in\langle 5\rangle\}$ where $N_s=\langle (\sigma,\delta_s)\rangle$.\par
\noindent We also note that, unlike the odd prime-power case, $T(C_{2^n})\cong C_{2}$ for $n\geq 3$.\par
\noindent As to $\mathcal{Q}=\mathcal{S}\cap\mathcal{R}$ we consider when  $N_s$ is normalized by $G=\langle(\sigma,I)\rangle$. A straightforward calculation yields the following
\begin{proposition}
\label{2Q}
For $G=\langle\sigma\rangle\cong C_{2^n}$, for $n\geq 3$, if $r=[\frac{n-3}{2}]$ then
$$
\mathcal{Q}(G)=\{N_s\ |\ s\in\langle 5^{2^{r}})\rangle\}
$$
and $\mathcal{H}(G)=\{N_1,N_{5^{2^{n-3}}}\}$.
\end{proposition}
What we find then is that, for $n\geq 3$, $|\mathcal{Q}(G)|=2^{[\frac{n}{2}]}$.\par 
\noindent As to the (possible) group structure on a $\pi(\mathcal{Q}(G))$ we can formulate an 'explicit' formulation of the generators of each $N\in\mathcal{Q}(C_{p^n})$. Going forward we shall study $\mathcal{Q}(C_{p^n})$ for $C_{p^n}$ embedded in the ambient symmetric group $S_{p^n}$ as the regular subgroup $\langle(1,\dots,p^n)\rangle$ and look at regular subgroups of $\operatorname{Hol}(C_{p^n})$ that are also cyclic of order $p^n$. We will be guided, initially, by the fact that we know how many groups reside in $\mathcal{Q}(C_{p^n})$ for $p=2$ and $p>2$ although there will be some differences to the formulation of $\pi(\mathcal{Q}(C_{p^n}))$ for $p>2$ versus $p=2$.
\begin{definition}
If $p$ is prime and $\sigma=(1,\dots,p^n)$ then $\sigma^{p^{n-m}}=\sigma_1\cdots\sigma_{p^{n-m}}$ where 
$$
\sigma_i=(i,i+p^{n-m},\dots,i+(p^m-1)p^{n-m})
$$
for $i=1,\dots,p^{n-m}$ and define two other elements
\begin{align*}
\ds\gamma&=\prod_{i=1}^{p^{n-m}}\sigma_i^{i-1}\\
\ds\beta&=\ds\prod_{i=1}^{p^{n-m}}\sigma_i^{t_{i-1}}\\
\end{align*}
where $t_{j}=\frac{j(j+1)}{2}$.
\end{definition}
The element $\gamma$ is the key to constructing elements of $\mathcal{Q}(C_{p^n})$, and we begin with a number of facts about how it acts on $\sigma$.
\begin{lemma}
\label{uaut}
For $\sigma$, $\gamma$ as defined above, we have that $\gamma\in \operatorname{Aut}(\langle\sigma\rangle)$ and $\gamma\sigma\gamma^{-1}=\sigma^{u}$ for $u=1+p^{n-m}\in U_{p^n}$, for any prime $p$.
\end{lemma}
\begin{proof}
By direct calculation 
\begin{align*}
\gamma\sigma(i+kp^{n-m})&=\gamma(i+1+kp^{n-m})\\
                        &=i+1+(k+i)p^{n-m})\\
\sigma^{u}\gamma(i+kp^{n-m})&=\sigma^u(i+(k+(i-1))p^{n-m})\\
                            &=i+u+(k+(i-1))p^{n-m}
\end{align*}
and so $\gamma\sigma(i+kp^{n-m})=\sigma^{u}\gamma(i+kp^{n-m})$ for all $i,k$ provided $1=u-p^{n-m}$, that is $u=1+p^{n-m}$. The presentation of $\gamma$ being an element of $B=Perm({1,\dots,p^n})$ where $\langle\sigma\rangle$ is patently a regular subgroup implies that $\gamma\in \operatorname{Aut}(\langle\sigma\rangle)$ where we view $\operatorname{Aut}(\langle\sigma\rangle)$ as a subgroup of $\operatorname{Norm}_B(\langle\sigma\rangle)$. We note also that $\gamma(1)=1$ since $\sigma_1^{1-1}=()$. 
\end{proof}
\noindent We shall use $\gamma$ to construct $\mathcal{S}\cap\mathcal{R}$ but first we need the following technical fact to establish two properties about the unit $u$ above, which are provable by a routine induction argument.
\begin{lemma}
\label{valuation}
For $p$ prime, and $u=1+p^{n-m}\in U_{p^n}$, if $v_{p}(x)$ denotes the $p$-valuation of an integer $x$ we have:\par
\noindent (a) $v_{p}(u^k-1)=\begin{cases} m+v_{p}(k) \ \text{if $n$ even} \\ 
                                m+1+v_{p}(k) \ \text{if $n$ odd}
                  \end{cases}$, and\par
\noindent (b) $v_{p}(1+u^e+\dots+u^{e(t-1)})=v_{p}(t)$ for any positive integers $e,t$.\par
\noindent (c) $|u|=2^m$.
\end{lemma}
\noindent We can now determine the order of $\gamma^e\sigma$.
\begin{lemma}
For $p$ prime, and $e\in\mathbb{Z}_{p^m}$, one has $|\gamma^e\sigma|=p^n$.
\end{lemma}
\begin{proof}
We begin by observing that, since $\gamma\sigma\gamma^{-1}=\sigma^u$ then $\gamma^e\sigma\gamma^{-e}=\sigma^{u^e}$ which allows one to show that
$$
(\gamma^e\sigma)^t=\gamma\sigma^{1+u^e+\dots+u^{e(t-1)}}\gamma^{t-1}
$$
and since $\gamma\sigma^x\gamma^t=I$ if and only if $\sigma^x=\gamma^{-t}$ then $(\gamma^e\sigma)^t$ is the identity if and only if $\sigma^{1+u^e+\dots+u^{e(t-1)}}$ is the identity. However, by \ref{valuation} the $p$-valuation of the exponent is $n$ if and only if $t=n$.
\end{proof}
\begin{proposition}
\label{SRpn}
If $p$ is prime and $G=C_{p^n}=\langle\sigma\rangle$ then $\mathcal{S}(G)\cap\mathcal{R}(G)=\{\langle\gamma^e\sigma\rangle\ |\ e\in\mathbb{Z}_{p^m}\}$.
\end{proposition}
\begin{proof}
As seen above, $\gamma^e\in \operatorname{Aut}(G)$ and so $\gamma^e\sigma\in \operatorname{Hol}(G)$ and therefore $N_e\in\mathcal{S}(G)$. To show $N_e\in\mathcal{R}(G)$ we recall that $\gamma\sigma\gamma^{-1}=\sigma^u$ so that $\gamma^{-1}\sigma\gamma=\sigma^v$ where $uv=1$ (i.e. $v=1-p^{n-m}$) and so $\gamma^{-e}\sigma\gamma^{e}=\sigma^{v^e}$ which means
\begin{align*}
\sigma(\gamma^e\sigma)\sigma^{-1} &= \sigma\gamma^e\\
                                  &= \gamma^e\sigma^{v^e}\\
                                  &= (\gamma^e\sigma)^{v^e}
\end{align*}
and so we see that $N_e$ normalizes $G$. The other verification to make is that each $e\in\mathbb{Z}_{p^m}$ determines a distinct group. To this end, suppose that $\gamma^e\sigma=(\gamma^f\sigma)^v$ for $v\in U_{p^n}$. Since 
\begin{align*}
\gamma^e\sigma\sigma(\gamma^e\sigma)^{-1}&=\sigma^{u^e}\\
(\gamma^f\sigma)\sigma((\gamma^f\sigma)^v)^{-1}&=\sigma^{u^{fv}}\\
\end{align*}
then we must have $e\equiv fv(mod\ p^m)$. Now, we have that 
$$(\gamma^f\sigma)^v=\gamma\sigma^{1+u^f+\dots+u^{(v-1)f}}\gamma^{v-1}$$
so if $\gamma^{fv}\sigma=(\gamma^f\sigma)^v$ then we must have 
\begin{align*}
\gamma^{fv-1}\sigma &= \gamma\sigma^{1+u^f+\dots+u^{(v-1)f}}\gamma^{v-1}\\
                   &\downarrow\\
\sigma^{u^{fv-1}}\gamma^{fv-1}&=\sigma^{1+u^f+\dots+u^{(v-1)f}}\gamma^{v-1}\\
                   &\downarrow\\
\sigma^{u^{fv-1}}\gamma^{(f-1)v}&=\sigma^{1+u^f+\dots+u^{(v-1)f}}\\
                          &\downarrow\\
\gamma^{(f-1)v} &= \sigma^{1+u^f+\dots+u^{(v-1)f}-u^{fv-1}}\\                  
\end{align*}
where the last line implies that $f=1$. As such $1+u+\dots+u^{(v-2)}=0$. But by \ref{valuation} this implies that $v_p(v-1)=n$ which means $v-1=ap^n$ namely that $v=1$. 
\end{proof}
We now consider the structure of $\QHG$ where (unlike the facts established above) there is a difference between the parameter sets when $p=2$ versus $p>2$.
\begin{theorem} For $G=\langle \sigma\rangle$ for $p\geq 3$ and $n\geq 2$, or $p=2$ and $n=2m+1$ for $m\geq 1$, then $\mathcal{Q}(G)$ is parameterized by powers of $\beta$ defined above.
\end{theorem}
\begin{proof} Recall that
$$
\beta=\displaystyle \prod_{i=1}^{p^{n-m}} \sigma_i^{t_{i-1}} \text{ and }\gamma=\displaystyle \prod_{i=1}^{p^{n-m}} \sigma_i^{i-1}
$$
for $\sigma^{p^{n-m}}=\sigma_1\cdots\sigma_{p^{n-m}}$ where the support of each $p^m$-cycle consists of elements with the same remainder mod $p^{n-m}$ and $\sigma_i(i+kp^{n-m})=i+(k+1)p^{n-m}$. To show the parameterization asserted, we shall show that
$$
\beta^e\sigma\beta^{-e}=\gamma^e\sigma
$$
for each $e\in\mathbb{Z}_{p^m}$. The delicate part of this calculation is paying attention to the fact that applying $\beta$ or $\gamma$ to an element of $x\in\{1,\dots,p^n\}$ requires determining the unique $i$ such that $x\in Supp(\sigma_i)$. We begin by computing 
\begin{align*}
\gamma^e\sigma\beta^e(i+kp^{n-m})&=\gamma^e\sigma\sigma_i^{et_{i-1}}(i+kp^{n-m})\\
                                 &=\gamma^e\sigma(i+(k+et_{i-1})p^{n-m})\\
                                 &=\gamma^e(i+1+(k+et_{i-1})p^{n-m})\\
                                 &=\begin{cases} \sigma_{i+1}^{ei}(i+1+(k+et_{i-1})p^{n-m})\ i<p^{n-m} \\
                                                 \sigma_{1}^{e\cdot 0}(1+(1+k+et_{p^{n-m}-1})p^{n-m})\ i=p^{n-m} \\
                                   \end{cases}\\
                                 &=\begin{cases} i+1+(k+et_{i-1}+ei)p^{n-m})\ i<p^{n-m}\\
                                                 1+(1+k+et_{p^{n-m}-1})p^{n-m})\ i=p^{n-m} \\
                                   \end{cases}\\
                                 &\text{ as compared with }\\
\beta^e\sigma(i+kp^{n-m}) &= \beta^e(i+1+kp^{n-m})\\
                          &= \begin{cases} \sigma_{i+1}^{et_i}(i+1+kp^{n-m})\ i<p^{n-m}\\
                                           \sigma_{1}^{et_0}(1+(k+1)p^{n-m})\ i=p^{n-m}\\
                             \end{cases}\\
                          &= \begin{cases} i+1+(k+et_i)p^{n-m}\ \ i<p^{n-m}\\
                                           1+(k+1)p^{n-m}\ \ \ \ \ \ \ \ \ \ i=p^{n-m}\\
                             \end{cases}\\
\end{align*}
so we must check whether
\begin{align*}
i+1+(k+et_{i-1}+ei)p^{n-m})&\equiv i+1+(k+et_i)p^{n-m}(mod\ p^n)\ \ i<p^{n-m}\\
1+(1+k+et_{p^{n-m}-1})p^{n-m})&\equiv  1+(k+1)p^{n-m}(mod\ p^n)\  i=p^{n-m}\\
\end{align*}
for $i\in\{1,\dots,p^{n-m}\}$ and $e,k\in\mathbb{Z}_{p^m}$. For $i<p^{n-m}$ 
\begin{align*}
i+1+(k+et_{i-1}+ei)p^{n-m})&\equiv i+1+(k+et_i)p^{n-m}(mod\ p^n)\\
                           &\updownarrow\\
(et_{i-1}+ei)p^{n-m})&\equiv (et_i)p^{n-m}(mod\ p^n)\\
                           &\updownarrow\\
(et_{i-1}+ei)&\equiv (et_i)(mod\ p^m)\\
\end{align*}
where the last congruence holds since $t_{i-1}+i=t_i$. The case for $i=p^{n-m}$ is where we see the effect of the prime $p$ and $m=[\frac{n}{2}]$. Specifically
\begin{align*}
1+(1+k+et_{p^{n-m}-1})p^{n-m})&\equiv  1+(k+1)p^{n-m}(mod\ p^n)\\
                           &\updownarrow\\
(et_{p^{n-m}-1})p^{n-m})&\equiv  0(mod\ p^n)\\
     &\updownarrow\\
et_{p^{n-m}-1}&\equiv  0(mod\ p^m)\\
\end{align*}
where now, if $p>2$ then $t_{p^{n-m}-1}=\frac{(p^{n-m}-1)p^{n-m}}{2}=\frac{(p^{m}-1)p^{m}}{2}\equiv 0(mod\ p^m)$, regardless of $m$. If $p=2$ then 
\begin{align*}
t_{2^{n-m}-1}&=\frac{(2^{n-m}-1)2^{n-m}}{2}\\
             &=(2^{n-m}-1)2^{n-m-1}\\
             &=\begin{cases}
               (2^{m}-1)2^{m-1}\ \ \ n=2m\\
               (2^{m+1}-1)2^{m}\ \ \ n=2m+1
               \end{cases}
\end{align*}
and so $t_{2^{n-m}-1}\equiv 0(mod\ 2^{m})$ only if $n=2m+1$. We note that if $e$ is even, then $et_{p^{n-m}-1}\equiv 0(mod\ 2^m)$ regardless of $n$, so $\langle\beta^2\rangle$ parameterizes those $N_e\in\mathcal{Q}(C_{2^n})$ where $e$ is even. (The question is whether there exists an element of order $2^m$ which generates a $\pi(\mathcal{Q}(2^n))$.) 
As we established earlier, each group $\langle\gamma^e\sigma\rangle$ is distinct. We've just shown that for $p$ odd, and $n\geq 3$ or, $p=2$ and $n=2m+1$ for $m\geq 1$ that $\beta^e$ conjugates $\sigma$ to $\gamma^e\sigma$ which means $\langle\beta\rangle$ is a $\pi(\mathcal{Q}(G))$. 
\end{proof}
We have therefore shown the following about $\QHG$.
\begin{corollary}
\label{Q2}
For $p=2$ and $n=2m+1$, for $m\geq 1$, or $p>2$ prime (for any $n$) if $G=C_{p^n}$, the group $\QHG$ is a Zappa-Sz\'ep extension of $\operatorname{Hol}(G)$, namely $\pi(\mathcal{Q}(G))\operatorname{Hol}(G)$ where $\pi(\mathcal{Q}(G))=\langle\beta\rangle$.
\end{corollary} 
The $n=2m$ case for $p=2$ in the above proof is certainly rather curious in that $\mathcal{Q}(G)=\{N_e\ |\ e\in\mathbb{Z}_{2^m}\}$ but that for the '$\beta$' above, we find that $\beta G\beta^{-1}\not\in\mathcal{Q}(G)$ but that $Orb_{\langle\beta^2\rangle}(G)=\{N_e\ |\ e\in 2\mathbb{Z}_{2^m}\}$. The question then is what elements parameterize $N_e$ for $e\in 1+2\mathbb{Z}_{2^m}$? And can we infer the existence of a parameter set for $\pi(\mathcal{Q}(G))$ which is a group? We can give a partial answer to these questions by means of the following modest observation.
\begin{proposition} 
$\beta^{2k}\gamma\sigma\beta^{-2k}=\gamma^{2k+1}\sigma$
\end{proposition}
\begin{proof}
We already know that $\beta^{2k}\sigma\beta^{-2k}=\gamma^{2k}\sigma$ so $\beta^{2k}\gamma\sigma\beta^{-2k}=\gamma^{2k+1}\sigma$ if and only if $\beta^{2k}\gamma\beta^{-2k}\gamma^{2k}\sigma=\gamma^{2k+1}\sigma$. But this is equivalent to $\beta^{2k}\gamma\beta^{-2k}=\gamma$, but this is true by virtue of how $\beta$ and $\gamma$ are defined previously.
\end{proof}
This results in a kind of parameterization of those $N_e\in\mathcal{Q}(G)$ for $e$ odd.
\begin{corollary}
For $N_1=\langle\gamma\sigma\rangle$ one has that $Orb_{\langle\beta^2\rangle}(N_1)=\{N_e\ |\ e\in 1+2\mathbb{Z}_{2m}\}$
\end{corollary}
So for $p=2$ and $n=2m$ we have subgroups of $\QHG$ which parameterize the two 'halves' of $\mathcal{Q}(G)$. One way to infer the existence of a subgroup $\pi(\mathcal{Q}(G))$ would be to find an element $\alpha\in\QHG$ such that $\alpha^2=\beta^2$ such that $\alpha\sigma\alpha^{-1}=\gamma\sigma$. For $n$ odd, this is exactly $\beta$ itself, of course, but it hints at the fact that, for $n$ even, such an $\alpha$ would be a 'square root' of $\beta^2$. As such, for even $n$, given the cycle structure of $\beta^2$ it would become a computational question to find such an $\alpha$. From calculations done in GAP, which we shall explore more fully later, $\QHG$ does appear to be a Zappa-Sz\'ep extension for these cyclic $2$-groups, and that there exists $\pi(\mathcal{Q}(G))$ which are cyclic.\par
\section{Cyclic groups in general}
For general cyclic groups $C_n$ where  $n=p_1^{e_1}p_2^{e_2}\dots p_r^{e_r}$ for distinct primes $p_i$ then $\operatorname{Hol}(C_n)\cong \operatorname{Hol}(C_{p_1^{e_1}})\times\cdots\times \operatorname{Hol}(C_{p_r^{e_r}})$. The question is what is the relationship between $\mathcal{Q}(C_n)$ and $\mathcal{Q}(C_{p_i^{e_i}})$, and for the purpose of enumeration, between $|\mathcal{Q}(C_n)|$ and $|\mathcal{Q}(C_{{p_i}^e_i})|$.  We note that Miller in \cite{Miller1908} observed that $N\!\operatorname{Hol}(C_n)$ and concordantly $T(C_n)$ also splits along the relatively prime components. As such $|\mathcal{H}(G)|=|\mathcal{H}(C_{p_1^{e_1}})|\cdots|\mathcal{H}(C_{p_r^{e_r}})|$ where, of course, $|\mathcal{H}(C_{p_i^{e_i}})|=1$ if $p>2$.\par
We can show a considerably stronger result than just a statement about relatively prime order cyclic groups. 
\begin{proposition}
\label{QC}
If $|G_1|=n_1$ and $|G_2|=n_2$ where $gcd(n_1,n_2)=1$ then $|\mathcal{Q}(G_1\times G_2)|=|\mathcal{Q}(G_1)|\cdot|\mathcal{Q}(G_2)|$.
\end{proposition}
\begin{proof}
If $G=G_1\times G_2$ then $\lG\leq B=Perm(G)\cong S_{n_1n_2}$. Moreover, $\operatorname{Aut}(G)\cong \operatorname{Aut}(G_1)\times \operatorname{Aut}(G_2)$ and therefore $\operatorname{Hol}(G)\cong \operatorname{Hol}(G_1)\times \operatorname{Hol}(G_2)$. However, our argument will not explicitly require a consideration of the structure of $\operatorname{Hol}(G)$, but rather is focused on the mutual normalization property. By a slight abuse of notation, we may represent $\lG$ as an internal direct product $G_1G_2$ where $G_1$ and $G_2$ are semi-regular subgroups. \par
If now $N\in\mathcal{Q}(G)$ then similarly we may write $N=N_1N_2$ where $N_i\cong G_i$. Since $N\in\mathcal{Q}(G)$ then $N_1N_2$ normalizes $G_1G_2$ and is normalized by $G_1G_2$. What we can show is a relationship between the supports of the cycles that make up the elements of $G_i$ and those of $N_i$.\par
For $x\in G$ let $A_i=Orb_{G_i}(x)$ and $B_i=Orb_{N_i}(x)$ for $i=1,2$. Since each $G_i$ is characteristic in $G$ then for $\eta\in N_i$ we have $\eta G_i\eta^{-1}=G_i$. Thus, since $Orb_{\eta G_i \eta^{-1}}(x)=\eta(A_i)$ and $\eta G_i\eta^{-1}=G_i$, we have that $\eta(A_i)=A_i$ and since $x\in A_i$ then $\eta(x)\in A_i$ for {\it all} $\eta\in N_i$. But this means $B_i=Orb_{N_i}(x)\subseteq A_i$ and since $|A_i|=|G_i|=|N_i|=|B_i|$ we conclude that $A_i=B_i$. Furthermore, if $g=g_1g_2$ where $g_i=g_{i1}g_{i2}\cdots g_{i(n_1n_2/n_i)}\in G_i$ and $\eta=\eta_1\eta_2$ where $\eta_i=\eta_{i1}\eta_{i2}\cdots\eta_{i(n_1n_2/n_i)}\in N_i$ then we may arrange the $g_{ij}$ and $\eta_{ij}$ so that $Supp(g_{ij})=Supp(\eta_{ij})$ and $Supp(g_{ij})\cap Supp(\eta_{ik})=\emptyset$ for $j\neq k$. 
If $g\eta_1\eta_2 g^{-1}=\eta'_1\eta'_2$ where $\eta'_i=\eta'_{i1}\eta'_{i2}\cdots\eta'_{i(n_1n_2/n_i)}$ we may assume therefore that $g_{ij}\eta_{ij} g_{ij}^{-1}=\eta'_{ij}$ for $j=1\dots n/n_i$ and for $j\neq k$ that $g_{ij}\eta_{ik} g_{ij}^{-1}=\eta'_{ik}$. Since, $G$, whence the $G_i$ are fixed at the outset, it follows that this holds true for all $N\in\mathcal{Q}(G)$.\par
\noindent Suppose also that $\eta_1g_2\eta_1^{-1}=g_2'\in G_2$ and $g_2\eta_1^{-1}g_2^{-1}=\eta'_1\in N_1$ then 
\begin{align*}
\eta_1g_2\eta_1^{-1}g_2^{-1}&=g_2'g_2^{-1}\in G_2\\
\eta_1 g_2\eta_1^{-1}g_2^{-1}&=\eta_1^{-1}\eta'_1\in N_1\\
\end{align*}
which implies that $g'_2g_2^{-1}=\eta_1^{-1}\eta'_1$. But since $gcd(|G_2|,|N_1|)=gcd(n_2,n_1)=1$ then we must have that $g'_2g_2^{-1}=\eta_1^{-1}\eta'_1=id$ which implies that $g_2$ centralizes $\eta_1$, and a symmetric argument implies that $g_1$ centralizes $\eta_2$. What all these facts imply is that $\lG$ establishes a set of supports $S_{11},S_{12},\dots,S_{1n_2}$ for the cycles that make up $g_1$ and similarly supports $S_{21},S_{22},\dots,S_{2n_1}$ for the cycle structure of $g_2$. Note also that each $S_{1j}$ intersects $S_{2k}$ in exactly one point for each $j\in\{1,\dots,n_1\}$ and $k\in\{1,\dots,n_2\}$. Moreover, these same supports are the same for the cycles that make up $\eta_1$ and $\eta_2$. As such if $g=g_1g_2$ then the action of $\lG$ can be viewed as the action of $g_1$ simultaneously on each $S_{1i}$ and $g_2$ on each $S_{2i}$ where the action on the $S_{2i}$ permutes the $\{S_{2i}\}$ amongst each other in blocks, and these two actions are independent of each other. As such, if $N_1$ is any regular subgroup of $Perm(G_1)$ that normalizes $G_1$ and is normalized by $G_1$, where both are viewed as diagonally embedded in $Perm(S_{11}\times\cdots\times S_{1n_1})$ and if similarly we have $N_2$ a regular subgroup of $Perm(G_2)$ that normalizes $G_2$ and is normalized by $G_2$ then both can be embedded diagonally in $Perm(S_{21}\times\cdots\times S_{2n_2})$. As such $N_1\cdot N_2$ can be embedded as $N\leq Perm(G)$ where 
$$G=S_{11}\cup\dots\cup S_{1n_2}=S_{21}\cup\dots\cup S_{2n_1}$$
and is normalized by and normalizes $\lG=G_1G_2$. More generally, this holds for {\it any} pair of regular subgroups which mutually normalize each other, in particular the members of $\mathcal{Q}(G)$. Therefore the correspondence 
$$\mathcal{Q}(G_1)\times \mathcal{Q}(G_2)\ni(N_1,N_2)\mapsto N\in\mathcal{Q}(G_1\times G_2)$$
is bijective and establishes the count stated in the proposition.
\end{proof}
As to $\QHG$ for $G=G_1\times G_2$ a product of relatively prime order groups, if $\operatorname{QHol}(G_1)$, and $\operatorname{QHol}(G_2)$ are groups the argument used to show the bijection above, in particular the independence of the actions on the blocks versus the supports of blocks implies that $\QHG\cong \operatorname{QHol}(G_1)\times \operatorname{QHol}(G_2)$, which parallels the decomposability of $\HG$, $\NHG$ and $T(G)$ into direct products.\par
\section{Dihedral Groups}
Our first class of examples beyond cyclic groups will be the dihedral groups $D_n$ of order $2n$ for $n\geq 3$. The case of $n=3$ is relatively simple in that, since $D_3=S_3$ which is a complete group,  $\operatorname{Hol}(D_3)=\lambda(D_3)\times \rho(D_3)$ which easily yields that $\mathcal{S}=\mathcal{R}=\{\lambda(D_3),\rho(D_3)\}$. As such $\mathcal{Q}(D_3)=\mathcal{H}(D_3)$, and that $\NHG=\QHG\cong (S_3\times S_3)\rtimes C_2$ where the $C_2$ component is that element conjugating $\lambda(D_3)$ to $\rho(D_3)$, so that $\pi(\mathcal{Q}(G))$ is also this same subgroup of order $2$. Indeed, if we embed $D_3$ into $S_6$ then we have
\begin{align*}
\lambda(D_3)&=\langle (1,3)(2,5)(4,6), (1,4,5)(2,6,3) \rangle\\
\rho(D_3)&=\langle (1,2)(3,5)(4,6), (1,4,5)(2,3,6) \rangle\\
Q\!\operatorname{Hol}(D_3)&=\operatorname{Hol}(S_3)\langle\tau\rangle\\
&\text{ for any $\tau\in\{(1,4),(1,5),(2,3),(2,6),(3,6),(4,5)\}$}\\
\end{align*}
so that all $\pi(\mathcal{Q}(D_3))$ are isomorphic, which is unsurprising given that $\mathcal{Q}=\mathcal{H}$ and $Q\!\operatorname{Hol}(D_3)=N\!\operatorname{Hol}(S_3)$ so that any $\pi(\mathcal{Q}(G))$ would be isomorphic to $T(D_3)$. 
For $D_n=\langle x,t\ |\ x^n=1,t^2=1,xt=tx^{-1}\rangle$ in general, in \cite{Kohl-Dihedral-2020}, we have a complete enumeration of $\mathcal{R}(D_n)=R(D_n,[D_n])$.
\begin{theorem}\cite[4.4 4.7]{Kohl-Dihedral-2020}
\label{Kohl-Dihedral}
For the dihedral groups $D_n$ of order $2n$
$$|R(D_n,[D_n])| = \begin{cases}
                       (\frac{n}{2}+2)|\Upsilon_n| \ \ \text{ if }8|n\\
                       (\frac{n}{2}+1)|\Upsilon_n| \ \ \text{ if }4|n \text{ but }8\nmid n\\
                       (n+1)|\Upsilon_n| \ \ \text{ if }2|n\text{ but }4\nmid n\\
                       |\Upsilon_n| \ \ \ \ \ \ \ \ \ \ \ \text{ if }n\text{ odd }\\
\end{cases}$$
where $\Upsilon_n=\{u\in U_n\ |\ u^2=1\}$.
\end{theorem}
As such, we also, in principle, have an enumeration of $\mathcal{S}(D_n)\cap\mathcal{R}(D_n)$ by filtering those $N\in R(D_n,[D_n])$ which normalize $\lambda(D_n)$. In fact, we can determine $|\mathcal{Q}(D_n)|$ but we have to separate the analysis between $n$ being even vs. odd.\par
\noindent For $n$ odd, we can apply \cite[2.11]{Kohl-Multiple-2014} mentioned earlier, namely that $|\mathcal{H}(D_n)|=|\Upsilon_n|$. 
\begin{theorem}
For $n$ odd, $\mathcal{Q}(D_n)=\mathcal{H}(D_n)$.
\end{theorem}
\begin{proof}
For any group $G$ we have 
$$\mathcal{H}(G)\subseteq\mathcal{Q}(G)\subseteq \mathcal{S}(G)\cap\mathcal{R}(G)\subseteq\mathcal{R}(G)$$ 
so for $D_n$ (when $n$ is odd) we note that $|R(D_n,[D_n])|=|\Upsilon_n|=|\mathcal{H}(D_n)|$.
\end{proof}
For $n$ even, we see above that $|R(D_n,[D_n])|$ is dependent on what power of $2$ divides $n$. For $n=4$ we have $\Upsilon_4=\{\pm 1\}$ so that $\mathcal{H}(D_4)=\{\lambda(D_4),\rho(D_4)\}$ and from the above, that $|\mathcal{R}(D_4)|=6$ overall. It turns out to be the case that $\mathcal{S}(D_4)=\mathcal{R}(D_4)=\mathcal{Q}(D_4)$. Specifically, we have
\begin{align*}
\mathcal{Q}(D_4)&=\{\\
\langle (1,x,x^2,x^3)&(t,tx^3,tx^2,tx), (1,t)(x,tx)(x^2,tx^2)(x^3,tx^3) \rangle \leftarrow\lambda(D_4)\\
\langle (1,x,x^2,x^3)&(t,tx,tx^2,tx^3), (1,t)(x,tx^3)(x^2,tx^2)(tx,x^3) \rangle \leftarrow\rho(D_4)\\
\langle (1,tx,x^2,tx^3)&(t,x^3,tx^2,x), (1,t)(x,tx)(x^2,tx^2)(x^3,tx^3) \rangle\\
\langle (1,tx,x^2,tx^3)&(t,x,tx^2,x^3), (1,t)(x,tx^3)(x^2,tx^2)(tx,x^3) \rangle\\
\langle (1,t,x^2,tx^2)&(x,tx,x^3,tx^3), (1,x)(t,tx^3)(x^2,x^3)(tx,tx^2) \rangle\\
\langle (1,t,x^2,tx^2)&(x,tx^3,x^3,tx), (1,x)(t,tx)(x^2,x^3)(tx^2,tx^3) \rangle\}\\
\end{align*}
and one can see fairly easily that one choice of $\pi(\mathcal{Q}(D_4))$ is 
$$\langle(x,tx)(x^3,tx^3), (t,tx,x)(tx^2,tx^3,x^3)\rangle$$
and that $\pi(\mathcal{Q}(G))\cong S_3$. Moreover, one may show that 
$$\langle ( 1,t,x ) ( x^2,tx^2,x^3 ) ( tx,tx^3 )\rangle$$
also parameterizes $\mathcal{Q}$ but that here $\pi(\mathcal{Q}(G))\cong C_6$! It turns out that, in fact, there are 64 choices of subgroups of $\QHG$ that parameterize $\mathcal{Q}$, and that the isomorphism classes of these are evenly divided between $S_3$ and $C_6$.\par
We also note that, unlike cyclic groups, and the $D_4$ example above, we cannot expect $\mathcal{Q}(G)=\mathcal{S}(G)\cap\mathcal{R}(G)$. Indeed, for larger (even) $n$, we see that (generally) $|\mathcal{Q}(G)|<|\mathcal{S}(G)\cap\mathcal{R}(G)|$ and we shall examine these in more detail in the next section. 
To look at the $D_n$ for $n$ even case in a bit more detail as we need to dig deeper into the enumeration of $\mathcal{R}(D_n)$ in \cite{Kohl-Dihedral-2020} mentioned above. To avoid excess detail we highlight a few of the essential components of this enumeration.\par
For $D_n=\langle x,t\ |\ x^n=1,t^2=1,xt=tx^{-1}\rangle$ the left regular representation is generated by
\begin{align*}
\lx&=(1,x,x^2,\dots,x^{n-1})(t,\ tx^{n-1},\ \dots, tx)\\
\lt&=(1,\ t)(x,\ tx)\cdots (x^{n-1},\ tx^{n-1})\\
\end{align*}
and, similar to the analysis done in \ref{QC}, the generators $\lx$ and $\lt$ gives rise to 'blocks' that are preserved. And for odd $n$ one has 
\begin{align*}
X_0&=\{1,x,\dots,x^{n-1}\}\\
Y_0&=\{t,tx,\dots,tx^{n-1}\}\\
\end{align*}
while for even $n$ there are two others
\begin{align*}
X_1&=\{1,x^2,\dots,x^{n-2},t,tx^2,\dots,tx^{n-2}\}\\
Y_1&=\{x,x^3,\dots,x^{n-1},tx,tx^3,\dots,tx^{n-1}\}
\end{align*}
and
\begin{align*}
X_2&=\{1,x^2,\dots,x^{n-2},tx,tx^3,\dots,tx^{n-1}\}\\
Y_2&=\{x,x^3,\dots,x^{n-1},t,tx^2,\dots,tx^{n-2}\}
\end{align*}
which are preserved, and this is due to the correspondence between these sets and the number of index 2 subgroups of $D_n$. For $n$ odd, there is just one such subgroup, but for $n$ even there are three, but only one of which is characteristic.\par
And in general, any regular subgroup isomorphic to $D_n$ yields similar block structures with respect to the action of its two generators. What turns out to be the case though is that for any $N$ in $\mathcal{S}(D_n)$ or $\mathcal{R}(D_n)$, $N$ must also give rise to the {\it same} collection of blocks as does $\lambda(D_n)$. Moreover, the unique characteristic subgroup of index 2 of any such $N$ corresponds to exactly one of the possible blocks, $\{X_0,Y_0\}$, $\{X_1,Y_1\}$, or $\{X_2,Y_2\}$. And so the enumeration of $\mathcal{R}(D_n)$ is stratified according to which of these three blocks, corresponds to $N$'s unique characteristic subgroup.\par
When $n$ is odd, the only block structure is $\{X_0,Y_0\}$ and, as seen above, $\mathcal{R}(D_n)=\mathcal{S}(D_n)=\mathcal{H}(D_n)=\mathcal{Q}(D_n)$. But for $n$ even, by \cite[4.7]{Kohl-Dihedral-2020} the $N\in\mathcal{R}(D_n)$ are distributed between different block types, where the exact distribution is based on whether $2$, $4$, or $8$ divide $n$. For $n=6$ for example, $|\mathcal{R}(D_6)|=14$ overall, where $2$ correspond to $\{X_0,Y_0\}$ which constitute $\mathcal{H}(D_6)$ actually, and $6$ corresponding to $\{X_1,Y_1\}$ and $6$ corresponding to $\{X_2,Y_2\}$.\par
To show which of the containments $\mathcal{H}(D_n)\subseteq \mathcal{Q}(D_n)\subseteq\mathcal{S}(D_n)\cap\mathcal{R}(D_n)$ are proper for $D_n$ when $n$ is even, we start by recalling \ref{HinQ}, namely that $\mathcal{H}(G)\subseteq\mathcal{Q}(G)$ and so, in particular $\rho(G)\in\mathcal{Q}(G)$. And so, by the very definition of $\mathcal{Q}(G)$, if a given $N\in\mathcal{S}(G)\cap\mathcal{R}(G)$ is {\it not} normalized by $\rho(G)$ then $N\not\in\mathcal{Q}(G)$. If a given $N$ is such that its characteristic subgroup of order $n$ corresponds to the block structure $\{X_i,Y_i\}$ then we say $N\in R(G,[D_n];W(X_i,Y_i))$. 
\begin{proposition}
\label{QX0}
If $n>4$ is even, and $N\in R(G,[D_n];W(X_i,Y_i))$ for $i=1,2$ then $N$ is not normalized by $\rho(D_n)$.
\end{proposition}
\begin{proof}
If $N\in R(D_n,[D_n];W(X_1,Y_1))$ then its characteristic subgroup of order $n$ is generated by a product of disjoint $n$-cycles, $k_Xk_Y$ where $Supp(k_X)=X_i$ and $Supp(k_Y)=Y_i$. In particular, by \cite[4.7]{Kohl-Dihedral-2020} one has that
\begin{align*}
k_X&=(t^{a_0}x^{b_0},\dots,t^{a_{n-1}}x^{b_{n-1}})\\
k_Y&=(t^{c_0}x^{d_0},\dots,t^{c_{n-1}}x^{d_{n-1}})\\
\end{align*}
where $(a_0,b_0)=(0,0)$, $(c_0,d_0)=(0,1)$, $(a_r,b_r)=(1,0)$, $(c_0,d_0)=(0,1)$, $(c_s,d_s)=(1,1)$ for some $r,s\in\mathbb{Z}_n-\langle 2\rangle$ where $s=(r-2)v$ for $v\in\Upsilon_n$. Moreover the following difference equations hold: 
\begin{align*}
a_{2e}&=0& a_{2e+1}&=1\\
c_{2e}&=0& c_{2e+1}&=1\\
b_{2e}&=2e& d_{2e}&=2ev+1\\
b_{r+2e}&=-2e& d_{(r-2)v+2e}&=-2ev+1\\    
\end{align*}
for $e\in\mathbb{Z}_{n/2}$ and we observe that all $b_k$ are even and all $d_k$ are odd.\par
As $\rho(x)=(1,x^{n-1},\dots,x)(t,tx^{n-1},\dots,tx)$ then one has that $\rho(x)(X_i)=Y_i$ and $\rho(x)(Y_i)=X_i$ for $i=1,2$ since $\rho(x)(x^b)=x^{b-1}$ and $\rho(x)(tx^b)=tx^{b-1}$. As such, if $\rho(x)$ normalizes $N$ then 
$$\rho(x)k_Xk_Y\rho(x)^{-1}=(\rho(x)k_X\rho(x)^{-1})(\rho(x)k_Y\rho(x)^{-1})=k_Y^{q}k_X^{q}$$
for some $q\in U_n$. By a similar argument to that in \cite[4.7]{Kohl-Dihedral-2020} one has that $q\in\Upsilon_n$, therefore 
\begin{align*}
(t^{a_0}x^{b_0-1},\dots,t^{a_{n-1}}x^{b_{n-1}-1})&=(t^{c_0}x^{d_0},t^{c_q}x^{d_q},\dots,t^{c_{(n-1)q}}x^{d_{(n-1)q}})\\
(t^{c_0}x^{d_0-1},\dots,t^{c_{n-1}}x^{d_{n-1}-1})&=(t^{a_0}x^{b_0},t^{a_q}x^{b_q}\dots,t^{a_{(n-1)q}}x^{b_{(n-1)q}})
\end{align*}
and since $(c_0,d_0-1)=(0,0)=(a_0,b_0)$ then $d_k-1=b_{qk}$ for each $k\in\{0,\dots,n-1\}$. So, for $k=2$ we have $d_2-1=b_{2q}$ where $d_2=2v+1$ and $b_{2q}=2q$ which means $2v=2q$. Similarly $d_{(r-2)v}-1=b_{(r-2)vq}$ where $d_{(r-2)v}=1$ and so $b_{(r-2)vq}=0$. However, we've just seen that $2v=2q$ so $(r-2)vq=rvq-2vq=rvq-2$ which means $b_{(rvq-2)}=0$. However, only $b_0$ and $b_{r}$ are zero which means either $rvq-2=0$ which is impossible since $r$, $v$, and $q$ are odd, or $rqv-2=r$ which means $r(qv-1)=2$. For $n=4$ this permits $(r,v,q)\in\{(1,1,3),(1,3,1),(3,1,3),(3,3,1)\}$, but for $n=4$ the parameters $(r,v)$ give rise to the same $k_Xk_Y$ for $(r,v)=(1,1)$ and $(3,1)$ and similarly one gets equal $k_Xk_Y$ for $(r,v)=(1,3)$ and $(3,3)$, which are exactly the two $N\in R(D_4,[D_4];W(X_1,Y_1))$. For {\it larger} even $n$ there are no $(r,v,q)$ that satisfy the conditions $r(qv-1)=2$ and $2q=2v$. And since each $M\in R(D_n,[D_n];W(X_2,Y_2))$ is the image of an $N\in R(D_n,[D_n];W(X_1,Y_1))$ under an automorphism of $D_n$ then $\rho(x)$ does not normalize any $M\in R(D_n,[D_n];W(X_2,Y_2))$ for $n>4$, although there are two in $R(D_4,[D_4];W(X_2,Y_2))$ which are the images under this automorphism of the two in $R(D_4,[D_4];W(X_1,Y_1))$.
\end{proof}
The other determination to make is what $N$ actually lie in $\mathcal{Q}(D_n)$.
\begin{proposition}
\label{QDn}
For $n>4$ even or odd, $\mathcal{Q}(D_n)=R(D_n,[D_n];W(X_0,Y_0))$.
\end{proposition}
\begin{proof}
As we saw above, $\mathcal{Q}(D_n)=\mathcal{H}(D_n)=R(D_n,[D_n])$ for $n$ odd, so it must be that $\mathcal{H}(D_n)= R(D_n,[D_n];W(X_0,Y_0))$. For $n$ even, we consider the nature of the order $n$ characteristic subgroup of any $N\in R(D_n,[D_n];W(X_0,Y_0))$. Any such $N$ has a characteristic subgroup of order $n$, generated by a product of disjoint $n$-cycles, $k_X$ and $k_Y$, where $Supp(k_X)=X$ and $Supp(k_Y)=Y$. For the case of $\{X,Y\}=\{X_0,Y_0\}$ we have, from \cite[4.4]{Kohl-Dihedral-2020}, that for any $n$ 
\begin{align*}
k_X&=(1,x,\dots,x^{n-1})\\
k_Y&=(t,tx^u,\dots,tx^{(n-1)u})
\end{align*}
for $u\in\Upsilon_n$, and, if $8|n$ an additional set, 
\begin{align*}
\tilde k_X&=(1,x^{i_{2v}-1},x^{i_{3v}-1}\dots,x^{i_{0}-1})\\
\tilde k_Y&=(t,tx^{i_{(1+u)v}-1},tx^{i_{(1+2u)v}-1},\dots,tx^{i_{(1+(n-1)u)v}-1})
\end{align*}
where $i_0=0$, $i_v=1$, and $i_{(1+e)v}-i_e=1$ for $e\in\mathbb{Z}_n$ where $v=\frac{n}{2}+1$, $u\in \Upsilon_n$. A bit of calculation further reveals that, in fact,  $i_{b+av}=a+b$ for $a=b$ or $a=b+1$ for $a,b\in\{0,\dots,\frac{n}{2}-1\}$. What one shows then is that the $k_Xk_Y$ all centralize each other, as do the $\tilde k_X\tilde k_Y$ mutually centralize each other. Additionally, one can show that $k_Xk_Y(\tilde k_X\tilde k_Y)(k_Xk_Y)^{-1}=(\tilde k_X\tilde k_Y)^v$ and vice versa.
\end{proof}
We note that the $k_Xk_Y$ in the above proof correspond exactly to $\mathcal{H}(D_n)$ which consists of $|\Upsilon_n|$ groups, and that for $8|n$ we have the additional $|\Upsilon_n|$ groups $\tilde k_X\tilde k_Y$ corresponding to the parameter $v=\frac{n}{2}+1$, as $u$ varies over $\Upsilon_n$.
\begin{corollary}
For the group $D_n$, for $n>4$ we have 
$$|\mathcal{Q}(D_n)|=\begin{cases} |\Upsilon_n|\ \text{if $8\ndiv n$}\\ 2|\Upsilon_n|\ \text{if $8|n$}\end{cases}$$
\end{corollary}
The other determination to make is whether there exists a $\pi(\mathcal{Q}(D_n))$ that is a group.
\begin{proposition}
For all $n$, there exists a $\pi(\mathcal{Q}(D_n))$ that is an elementary abelian $2$-group.
\end{proposition}
\begin{proof}
 As we observed above, when $8\ndiv n$ we have that $\mathcal{Q}(D_n)=\mathcal{H}(D_n)$. And for all $n$, by \cite{Kohl-Multiple-2014}, $N\!Hol(D_n)$ is a split extension of $Hol(D_n)$ by the group $M_n=\{\tau_u\ |\ u\in\Upsilon_n\}$ where $\tau_u(x^i)=x^i$ and $\tau_u(tx^i)=tx^{ui}$. So for $8\ndiv n$, we have that $Q\!Hol(D_n)=N\!Hol(D_n)$ where $M_n$ is a $\pi(\mathcal{Q}(D_n))$ and is an elementary abelian $2$-group.\par
When $8\div n$ we saw in \ref{QDn} that $\mathcal{Q}(D_n)$ consists of the $|\Upsilon_n|$ members of $\mathcal{H}(D_n)$, determined by their order $n$ characteristic subgroups:
\begin{align*}
k_X&=(1,x,\dots,x^{n-1})\\
k_Y&=(t,tx^u,\dots,tx^{(n-1)u})
\end{align*}
for each $u\in\Upsilon_n$, together with an addition $|\Upsilon_n|$ members with characteristic subgroups with the following generators:
\begin{align*}
\tilde k_X&=(1,x^{i_{2v}-1},x^{i_{3v}-1}\dots,x^{i_{0}-1})\\
\tilde k_Y&=(t,tx^{i_{(1+u)v}-1},tx^{i_{(1+2u)v}-1},\dots,tx^{i_{(1+(n-1)u)v}-1})
\end{align*}
where $v=\frac{n}{2}+1$ with the subscripts $i$ having been described above. What one can show is that the mapping
\begin{align*}
x^c&\mapsto x^{i_{(c+1)v}-1}\\
tx^{du}&\mapsto tx^{i_{(1+du)v}-1}\\
\end{align*}
conjugates each $k_X$ to $\tilde k_X$ and $k_Y$ to $\tilde k_Y$ and has order 2. And the union of these with $M_n$ described above yield a $\pi(\mathcal{Q}(D_n))$ which is a group, and indeed an elementary abelian $2$-group of order exactly twice that of $M_n$.
\end{proof}
The above illustrates when $\mathcal{H}(D_n)$ is a proper subset of $\mathcal{Q}(D_n)$. As to the containment $\mathcal{Q}(D_n)\subseteq \mathcal{S}(D_n)\cap\mathcal{R}(D_n)$, we observe that for $D_n$ ($n$ even) this containment is {\it always} proper. This is established by the following.
\begin{proposition}
For $n>4$, there exist $N\in R(D_n,[D_n];W(X_1,Y_1)$ that do not normalize $\lambda(D_n)$, that is $N\not\leq \operatorname{Hol}(D_n)$.
\end{proposition}
\begin{proof}
We refer back to the enumeration of the characteristic order $n$ subgroups $\langle k_Xk_Y\rangle$ of $N\in R(D_n,[D_n];W(X_1,Y_1)$ as described in \ref{QX0}. In particular $k_X=(t^{a_0}x^{b_0},\dots,t^{a_{n-1}}x^{b_{n-1}})$ and $k_Y=(t^{c_0}x^{d_0},\dots,t^{c_{n-1}}x^{d_{n-1}})$ where the exponents are subject to the difference equations also given \ref{QX0} parameterized by $r\in\mathbb{Z}_n-\langle 2\rangle$ and $v\in\Upsilon_n$. So suppose $N=N_{r,v}\in R(D_n,[D_n];W(X_0,Y_0)$ corresponds to a pair $(r,v)\in \mathbb{Z}_n-\langle 2\rangle\times \Upsilon_n$. If $4\ndiv n$ then $N_{r,v}$ normalizes $\langle\lambda(x)\rangle$ only if $v=1$. And if $4|n$ then $N_{r,v}$ normalizes $\langle\lambda(x)\rangle$ if $v=1,\frac{n}{2}+1$. And since $\Upsilon_n$ contains other units besides these (for $n>4$) in either case, the result follows.
\end{proof}
\section{ $|\mathcal{H}(G)|\leq |\mathcal{Q}(G)|\leq |\mathcal{S}(G)\cap\mathcal{R}(G)|$}
The two containments $\mathcal{H}(G)\subseteq \mathcal{Q}(G)\subseteq \mathcal{S}(G)\cap\mathcal{R}(G)$ may or may not be proper for a given group $G$. We have demonstrated the following:
\begin{align*}
|\mathcal{H}(C_{p^n})|&<|\mathcal{Q}(C_{p^n})|=|\mathcal{S}(C_{p^n})\cap\mathcal{R}(C_{p^n})|\text{\ \ \ $p$ odd}\\
|\mathcal{H}(C_{2^n})|&<|\mathcal{Q}(C_{2^n})|=|\mathcal{S}(C_{2^n})\cap\mathcal{R}(C_{2^n})|\text{  $1\leq n\leq 4$}\\
|\mathcal{H}(C_{2^n})|&<|\mathcal{Q}(C_{2^n})|<|\mathcal{S}(C_{2^n})\cap\mathcal{R}(C_{2^n})|\text{  $n\geq 5$}\\
|\mathcal{H}(D_n)|&=|\mathcal{Q}(D_n)|=|\mathcal{S}(D_n)\cap\mathcal{R}(D_n)|\text{ $n$ odd}\\
|\mathcal{H}(D_n)|&=|\mathcal{Q}(D_n)|<|\mathcal{S}(D_n)\cap\mathcal{R}(D_n)|\text{ $n$ even, $8\ndiv n$ }\\
|\mathcal{H}(D_n)|&<|\mathcal{Q}(D_n)|<|\mathcal{S}(D_n)\cap\mathcal{R}(D_n)|\text{ $n$ even, $8|n$ }\\
\end{align*}
As the $D_4$ example highlights, $Q\!\operatorname{Hol}(D_4)$ is a Zappa-Sz\'ep extension of $\operatorname{Hol}(D_4)$ with respect to complements $\pi(\mathcal{Q}(D_4))$ of different isomorphism classes, $C_6$ and $S_3$.\par
What we wish to do in this section is to give tables illustrating not only the disparity in sizes between $|\mathcal{H}(G)|$, $|\mathcal{Q}(G)|$, and $|\mathcal{S}(G)\cap\mathcal{R}(G)|$, but also indicate the different isomorphism classes of parameter groups $\pi(\mathcal{Q}(G))$ that may arise, and also to illustrate cases where $\QHG$ is not a Zappa-Sz\'ep extension of $\operatorname{Hol}(G)$. We include all groups of order at most 40, except for $n=32$, excluding $D_n$ for $n$ odd, and most cyclic groups. But we do include the cyclic groups $C_{4}$, $C_{16}$ and also $C_{64}$ so as to address the 'opening' left by \ref{Q2} regarding whether $\operatorname{QHol}(C_{2^n})$ is a group for $n$ even. We also include $D_n$ for $n$ even as we did not work out the size of $\mathcal{S}(D_n)\cap\mathcal{R}(D_n)$ in this case. We highlight (in gray) those cases where each of the two containments are proper. \par
We shall utilize GAP to compute these tables. Indeed, the examples of the two groups mentioned earlier of order 40 where $\NHG$ is not a split extension were also computed using GAP. And since we have used GAP to generate the following tables, we adopt the notational conventions they use, in particular writing $D_{2n}$ for the $n$-th dihedral group for example.\par
\hskip-1.2in\begin{tabular}{|c|c|c|c|c|}\hline
$G$ &  $|\mathcal{S}\cap\mathcal{R}|$ & $|\mathcal{Q}|$ & $|\mathcal{H}|$ & $\pi(\mathcal{Q})$ \\ \hline\hline
$C_{4}$ & $1$ & $1$ & $1$ & $1$ \\ \hline
$C_{2} \times C_{2}$ & $1$ & $1$ & $1$ & $1$ \\ \hline
$S_{3}$ & $2$ & $2$ & $2$ & $C_{2}$ \\ \hline
$C_{4} \times C_{2}$ & $8$ & $2$ & $2$ & $C_{2}$ \\ \hline
$D_{8}$ & $6$ & $6$ & $2$ & $S_{3}$, $C_{6}$ \\ \hline
$Q_{8}$ & $2$ & $2$ & $2$ & $C_{2}$ \\ \hline
$C_{2} \times C_{2} \times C_{2}$ & $8$ & $1$ & $1$ & $1$ \\ \hline
$C_{3} \times C_{3}$ & $9$ & $1$ & $1$ & $1$ \\ \hline
$C_{3} \rtimes C_{4}$ & $2$ & $2$ & $2$ & $C_{2}$ \\ \hline
$A_{4}$ & $6$ & $2$ & $2$ & $C_{2}$ \\ \hline
$D_{12}$ & $8$ & $2$ & $2$ & $C_{2}$ \\ \hline
$C_{6} \times C_{2}$ & $1$ & $1$ & $1$ & $1$ \\ \hline
$C_{16}$  & $4$ & $4$ & $2$ & $C_{4}$,$C_2\times C_2$\\ \hline
$C_{4} \times C_{4}$ & $24$ & $24$ & $1$ & $C_{4} \times S_{3}$, $(C_{6} \times C_{2}) \rtimes C_{2}$, $C_{3} \times D_{8}$, $S_{4}$, $C_{2} \times A_{4}$, $C_{2} \times C_{2} \times S_{3}$ \\ \hline
$(C_{4} \times C_{2}) \rtimes C_{2}$ & $76$ & $4$ & $4$ & $C_{2} \times C_{2}$ \\ \hline
$C_{4} \rtimes C_{4}$ & $72$ & $72$ & $8$ & $C_{3} \times S_{4}$, $(C_{3} \times A_{4}) \rtimes C_{2}$ \\ \hline
$C_{8} \times C_{2}$ & $10$ & $4$ & $4$ & $C_{2} \times C_{2}$ \\ \hline
$C_{8} \rtimes C_{2}$ & $10$ & $4$ & $4$ & $C_{2} \times C_{2}$ \\ \hline
$D_{16}$ & $16$ &\cg $8$ & $4$ & $C_{4} \times C_{2}$, $D_{8}$, $C_{2} \times C_{2} \times C_{2}$ \\ \hline
$QD_{16}$ & $32$ & $16$ & $16$ & $C_{2} \times D_{8}$ \\ \hline
$Q_{16}$ & $16$ &\cg $8$ & $4$ & $C_{4} \times C_{2}$, $D_{8}$, $C_{2} \times C_{2} \times C_{2}$ \\ \hline
$C_{4} \times C_{2} \times C_{2}$ & $146$ & $1$ & $1$ & $1$ \\ \hline
$C_{2} \times D_{8}$ & $198$ &\cg $6$ & $2$ & $S_{3}$, $C_{6}$ \\ \hline
$C_{2} \times Q_{8}$ & $66$ & $2$ & $2$ & $C_{2}$ \\ \hline
$(C_{4} \times C_{2}) \rtimes C_{2}$ & $224$ & $224$ & $2$ & $\QHG$ is not a Zappa-Sz\'ep Extension of $\operatorname{Hol}(G)$ \\ \hline
$C_{2} \times C_{2} \times C_{2} \times C_{2}$ & $106$ & $1$ & $1$ & $1$ \\ \hline
\end{tabular}
\par
\hskip-0.66in\begin{tabular}{|c|c|c|c|c|}\hline
$G$ &  $|\mathcal{S}\cap\mathcal{R}|$ & $|\mathcal{Q}|$ & $|\mathcal{H}|$ & $\pi(\mathcal{Q})$ \\ \hline\hline

$C_{3} \times S_{3}$ & $7$ & $2$ & $2$ & $C_{2}$ \\ \hline
$(C_{3} \times C_{3}) \rtimes C_{2}$ & $38$ & $2$ & $2$ & $C_{2}$ \\ \hline
$C_{6} \times C_{3}$ & $9$ & $1$ & $1$ & $1$ \\ \hline
$C_{5} \rtimes C_{4}$ & $2$ & $2$ & $2$ & $C_{2}$ \\ \hline
$C_{5} \rtimes C_{4}$ & $7$ & $2$ & $2$ & $C_{2}$ \\ \hline
$D_{20}$ & $12$ & $2$ & $2$ & $C_{2}$ \\ \hline
$C_{10} \times C_{2}$ & $1$ & $1$ & $1$ & $1$ \\ \hline
$C_{7} \rtimes C_{3}$ & $9$ & $2$ & $2$ & $C_{2}$ \\ \hline
$C_{3} \rtimes C_{8}$ & $4$ & $4$ & $4$ & $C_{2} \times C_{2}$ \\ \hline
$SL(2,3)$ & $6$ & $2$ & $2$ & $C_{2}$ \\ \hline
$C_{3} \rtimes Q_{8}$ & $16$ & $4$ & $4$ & $C_{2} \times C_{2}$ \\ \hline
$C_{4} \times S_{3}$ & $32$ & $8$ & $8$ & $C_{2} \times C_{2} \times C_{2}$ \\ \hline
$D_{24}$ & $16$ & $4$ & $4$ & $C_{2} \times C_{2}$ \\ \hline
$C_{2} \times (C_{3} \rtimes C_{4})$ & $20$ & $4$ & $4$ & $C_{2} \times C_{2}$ \\ \hline
$(C_{6} \times C_{2}) \rtimes C_{2}$ & $32$ & $8$ & $8$ & $C_{2} \times C_{2} \times C_{2}$ \\ \hline
$C_{12} \times C_{2}$ & $8$ & $2$ & $2$ & $C_{2}$ \\ \hline
$C_{3} \times D_{8}$ & $6$ & $6$ & $2$ & $S_{3}$, $C_{6}$ \\ \hline
$C_{3} \times Q_{8}$ & $2$ & $2$ & $2$ & $C_{2}$ \\ \hline
$S_{4}$ & $5$ & $2$ & $2$ & $C_{2}$ \\ \hline
$C_{2} \times A_{4}$ & $9$ & $2$ & $2$ & $C_{2}$ \\ \hline
$C_{2} \times C_{2} \times S_{3}$ & $44$ & $2$ & $2$ & $C_{2}$ \\ \hline
$C_{6} \times C_{2} \times C_{2}$ & $8$ & $1$ & $1$ & $1$ \\ \hline
$C_{5} \times C_{5}$ & $25$ & $1$ & $1$ & $1$ \\ \hline
$C_{9} \times C_{3}$ & $33$ &\cg $3$ & $1$ & $C_{3}$ \\ \hline
$(C_{3} \times C_{3}) \rtimes C_{3}$ & $78$ & $2$ & $2$ & $C_{2}$ \\ \hline
$C_{9} \rtimes C_{3}$ & $63$ &\cg $6$ & $2$ & $S_{3}$, $C_{6}$ \\ \hline
$C_{3} \times C_{3}\times C_{3}$ & $339$ & $1$ & $1$ & $1$ \\ \hline
$C_{7} \rtimes C_{4}$ & $2$ & $2$ & $2$ & $C_{2}$ \\ \hline
$D_{28}$ & $16$ & $2$ & $2$ & $C_{2}$ \\ \hline
$C_{14} \times C_{2}$ & $1$ & $1$ & $1$ & $1$ \\ \hline
$C_{5} \times S_{3}$ & $2$ & $2$ & $2$ & $C_{2}$ \\ \hline
$C_{3} \times D_{10}$ & $2$ & $2$ & $2$ & $C_{2}$ \\ \hline
\end{tabular}
\par
\hskip-0.86in\begin{tabular}{|c|c|c|c|c|}\hline
$G$ &  $|\mathcal{S}\cap\mathcal{R}|$ & $|\mathcal{Q}|$ & $|\mathcal{H}|$ & $\pi(\mathcal{Q})$ \\ \hline\hline
$C_{9} \rtimes C_{4}$ & $2$ & $2$ & $2$ & $C_{2}$ \\ \hline
$(C_{2} \times C_{2}) \rtimes C_{9}$ & $18$ &\cg $6$ & $2$ & $S_{3}$, $C_{6}$ \\ \hline
$D_{36}$ & $20$ & $2$ & $2$ & $C_{2}$ \\ \hline
$C_{18} \times C_{2}$ & $3$ & $3$ & $1$ & $C_{3}$ \\ \hline
$C_{3} \times (C_{3} \rtimes C_{4})$ & $7$ & $2$ & $2$ & $C_{2}$ \\ \hline
$(C_{3} \times C_{3}) \rtimes C_{4}$ & $38$ & $2$ & $2$ & $C_{2}$ \\ \hline
$C_{12} \times C_{3}$ & $9$ & $1$ & $1$ & $1$ \\ \hline
$(C_{3} \times C_{3}) \rtimes C_{4}$ & $11$ & $2$ & $2$ & $C_{2}$ \\ \hline
$S_{3} \times S_{3}$ & $55$ &\cg $4$ & $2$ & $C_{4}$, $C_{2} \times C_{2}$ \\ \hline
$C_{3} \times A_{4}$ & $42$ &\cg $6$ & $2$ & $S_{3}$, $C_{6}$ \\ \hline
$C_{6} \times S_{3}$ & $19$ & $2$ & $2$ & $C_{2}$ \\ \hline
$C_{2} \times ((C_{3} \times C_{3}) \rtimes C_{2})$ & $56$ & $2$ & $2$ & $C_{2}$ \\ \hline
$C_{6} \times C_{6}$ & $9$ & $1$ & $1$ & $1$ \\ \hline
$C_{13} \rtimes C_{3}$ & $15$ & $2$ & $2$ & $C_{2}$ \\ \hline
$C_{5} \rtimes C_{8}$ & $4$ & $4$ & $4$ & $C_{2} \times C_{2}$ \\ \hline
$C_{5} \rtimes C_{8}$ & $14$ & $4$ & $4$ & $\QHG$ is not a Zappa-Sz\'ep Extension of $\operatorname{Hol}(G)$ \\ \hline
$C_{5} \rtimes Q_{8}$ & $24$ & $4$ & $4$ & $C_{2} \times C_{2}$ \\ \hline
$C_{4} \times D_{10}$ & $48$ & $8$ & $8$ & $C_{2} \times C_{2} \times C_{2}$ \\ \hline
$D_{40}$ & $24$ & $4$ & $4$ & $C_{2} \times C_{2}$ \\ \hline
$C_{2} \times (C_{5} \rtimes C_{4})$ & $28$ & $4$ & $4$ & $C_{2} \times C_{2}$ \\ \hline
$(C_{10} \times C_{2}) \rtimes C_{2}$ & $48$ & $8$ & $8$ & $C_{2} \times C_{2} \times C_{2}$ \\ \hline
$C_{20} \times C_{2}$ & $8$ & $2$ & $2$ & $C_{2}$ \\ \hline
$C_{5} \times D_{8}$ & $6$ & $6$ & $2$ & $S_{3}$, $C_{6}$ \\ \hline
$C_{5} \times Q_{8}$ & $2$ & $2$ & $2$ & $C_{2}$ \\ \hline
$C_{2} \times (C_{5} \rtimes C_{4})$ & $24$ & $4$ & $4$ & $\QHG$ is not a Zappa-Sz\'ep Extension of $\operatorname{Hol}(G)$ \\ \hline
$C_{2} \times C_{2} \times D_{10}$ & $68$ & $2$ & $2$ & $C_{2}$ \\ \hline
$C_{10} \times C_{2} \times C_{2}$ & $8$ & $1$ & $1$ & $1$ \\ \hline
$C_{64}$   & $8$ & $8$ & $2$ & $C_{8}$, $Q_8$ \\ \hline 
\end{tabular}\par
We observe many examples in the above tables where $\mathcal{Q}(G)$ is a proper subset of $\mathcal{S}(G)\cap\mathcal{R}(G)$ and that it is frequently (but not always) the case that $\mathcal{Q}(G)=\mathcal{H}(G)$. We also note the intriguing entry for the group $(C_4\times C_2)\rtimes C_2$, of order $16$ (number 13 in the AllSmallGroups(16) list) for which $\QHG$ is not a Zappa-Sz\'ep product, even though $\NHG$ {\it is} a Zappa-Sz\'ep product, and, in fact, a semi-direct product. What is also unexpected is the size, which is much larger than $|\mathcal{Q}(G)|$ for other low order groups, including those of the same order. We note, of course, that there are instances where $\QHG=\NHG$ where $\NHG$ is not a Zappa-Sz\'ep extension, so, in fact, not a split extension of $\operatorname{Hol}(G)$.
\section{Further Questions}
We close with some questions which are yet unanswered, some of which are prompted by the data in the above tables. The primary question, as mentioned earlier, is that when $\mathcal{Q}(G)$ is a proper subset of $\mathcal{S}(G)\cap\mathcal{R}(G)$, does there always exist a $\pi(\mathcal{Q}(G))$ which is conj-closed, i.e. Is $\QHG$ always a group. As posited earlier, it is our conjecture that this is indeed always true. Beyond this, we pose the following questions.\par
\begin{itemize}
\item For abelian groups $G$, when is $\mathcal{Q}(G)=\{\lG\}$?
\item In particular, for elementary abelian groups $G$, is $\mathcal{Q}(G)=\{\lG\}$?
\item When is $\mathcal{Q}(G)=\mathcal{S}(G)\cap\mathcal{R}(G)$ ?
\item When is $\mathcal{Q}(G)$ properly larger vs. equal to $\mathcal{H}(G)$?
\item What are the conditions which make $\QHG$ a Zappa-Sz\'ep product?
\item When $\QHG$ is a Zappa-Sz\'ep product, what are the possible isomorphism classes of $\pi(\mathcal{Q}(G))$?
\item As $\mathcal{Q}(G)\subseteq R(\lambda(G),[G])$ then what are the implications for Hopf-Galois extensions arising from $N\in\mathcal{Q}(G)$ and also the braces $B$ with additive and circle groups both isomorphic to $G$? 
\end{itemize}
\bibliography{quasi}
\bibliographystyle{plain}
\end{document}